\documentclass[11pt]{amsproc}
\usepackage{amsmath}
\usepackage{amsfonts}

\theoremstyle{plain}

\newtheorem{corollary}{Corollary}

\newtheorem{lemma}{Lemma}

\newtheorem{proposition}{Proposition}

\newtheorem{theorem}{Theorem}
\numberwithin{equation}{section}
\input{tcilatex}

\begin{document}
\title[Elliptic Singular Fourth Order Equations]{Elliptic Singular Fourth
Order Equations}
\author{Mohammed Benalili}
\address{Faculty of Sciences, Dept. Math. University Aboubakr BelKa\"{\i}d.
Tlemcen Algeria.}
\email{m\_benalili@mail.univ-tlemcen.dz}
\author{Kamel Tahri}
\curraddr{Faculty of Sciences, Dept. Math. University Aboubakr Belka\"{\i}d.
Tlemcen. Algeria}

\begin{abstract}
Using a method developped in [1] and [2], we prove the existence of weak non
trivial solutions to fourth order elliptic equations with singularities and
with critical Sobolev growth.
\end{abstract}

\maketitle

\section{Introduction}

Fourth order elliptic equations have been widely studied these last years
because of their importance in the analysis on manifolds particularly those
involving the Paneitz - Branson operators. Many works have been devoted to
this subject ( see \cite{1}, \cite{2}, \cite{3}, \cite{4},\cite{5}, \cite{6}%
, \cite{7}, \cite{8}, \cite{9} \cite{10}, \cite{13} and \cite{16} ).
Different techniques have been used for the resolution of the fourth order
equations as example the variational method which was developed by Yamabe to
solve the problem of the prescribed scalar curvature. Let $(M,g)$ a compact
smooth Riemannian of dimension $n\geq 5$ with a metric $g$. We denote by $%
H_{2}^{2}(M)$ the standard Sobolev space which is the completed of the space 
$C^{\infty }\left( M\right) $ with respect to the norm 
\begin{equation*}
\left\Vert \varphi \right\Vert _{2,2}=\dsum\limits_{k=0}^{k=2}\left\Vert
\nabla ^{k}\varphi \right\Vert _{2}\text{.}
\end{equation*}%
$H_{2}^{2}(M)$ will be endowed with the suitable equivalent norm

\begin{equation*}
\left\Vert u\right\Vert _{H_{2}^{2}(M)}=(\int_{M}\left( \left( \Delta
_{g}u\right) ^{2}+\left\vert \nabla _{g}u\right\vert ^{2}+u^{2}\right)
dv_{g})^{\frac{1}{2}}\text{.}
\end{equation*}%
In 1979, \cite{17}, M. Vaugon has proved the existence of real $\lambda >0$
and a non trivial solution $u\in C^{4}\left( M\right) $ to the equation%
\begin{equation*}
\Delta _{g}^{2}u-div_{g}\left( a(x)\nabla _{g}u\right) +b(x)u=\lambda f(t,x)
\end{equation*}%
where $a$, $b$ are smooth functions on $M$ and $f(t,x)$ is odd and
increasing function in $t$ fulfilling the inequality 
\begin{equation*}
\left\vert f(t,x)\right\vert <a+b\left\vert t\right\vert ^{\frac{n+4}{n-4}}%
\text{.}
\end{equation*}%
D.E. Edminds, D. Fortunato and E. Jannelli \cite{14} have shown that the
only solutions in $R^{n}$ to the equation 
\begin{equation*}
\Delta ^{2}u=u^{\frac{n+4}{n-4}}
\end{equation*}%
are positive, symmetric, radial and decreasing functions of the form%
\begin{equation*}
u_{\epsilon }(x)=\frac{\left( (n-4)n(n^{2}-4)\epsilon ^{4}\right) ^{\frac{n-4%
}{8}}}{(r^{2}+\epsilon ^{2})^{\frac{n-4}{2}}}\text{.}
\end{equation*}%
In 1995, \cite{15} Van Der Vorst obtains the same results as D.E. Edminds,
D. Fortunato and E. Jannelli to the following problem 
\begin{equation*}
\left\{ 
\begin{array}{c}
\Delta ^{2}u-\lambda u=u\left\vert u\right\vert ^{\frac{8}{n-4}}\text{ in }%
\Omega , \\ 
\Delta u=u=0\text{ on }\partial \Omega ,%
\end{array}%
\right.
\end{equation*}%
where $\Omega $ is a bounded domain of $R^{n}$.\newline
In 1996, \cite{9} F. Bernis, J. Garcia-Azorero and I.Peral have obtained the
existence at least of two positive solution to the following problem%
\begin{equation*}
\left\{ 
\begin{array}{c}
\Delta ^{2}u-\lambda u\left\vert u\right\vert ^{q-2}=u\left\vert
u\right\vert ^{\frac{8}{n-4}}\text{ in }\Omega , \\ 
\Delta u=u=0\text{ on }\partial \Omega ,%
\end{array}%
\right.
\end{equation*}%
where $\Omega $ is bounded domain of $R^{n}$,$1<q<2$ and $\ \lambda >0$ in
some interval. In 2001, \cite{12}, D. Caraffa has obtained the existence of
a non trivial solution of class $C^{4,\alpha }$, $\ \alpha \in \left(
0,1\right) $ to the following equation 
\begin{equation*}
\Delta _{g}^{2}u-\nabla ^{\alpha }\left( a(x)\nabla _{\alpha }u\right)
+b(x)u=\lambda f(x)\left\vert u\right\vert ^{N-2}u
\end{equation*}%
with $\lambda >0$, first for $f$ a constant and next for a positive function
\ $f$ on $M$.

Recently the first author \cite{4}, has shown the existence of at least two
distinct non trivial solutions in the subcritical case and a non trivial
solution in the critical case to the following equation%
\begin{equation*}
\Delta _{g}^{2}u-\nabla ^{\alpha }\left( a(x)\nabla _{\alpha }u\right)
+b(x)u=f(x)\left\vert u\right\vert ^{N-2}u
\end{equation*}%
where $f$ is a changing sign smooth function and $a$ and $b$ are smooth
functions. In \cite{6} the same author proved the existence of at least two
non trivial solutions to%
\begin{equation*}
\Delta _{g}^{2}u-\nabla ^{\alpha }\left( a(x)\nabla _{\alpha }u\right)
+b(x)u=f(x)\left\vert u\right\vert ^{N-2}u+\left\vert u\right\vert
^{q-2}u+\varepsilon g(x)
\end{equation*}%
where $a$, $b$, $f$, $g$ are smooth functions on $M$ with $f>0$, $2<q<N$, $%
\lambda >0$ and $\epsilon >0$ small enough. Let $S_{g}$ denote the scalar
curvature of $M$. In 2010, \cite{8}, \ the authors proved the following
result

\begin{theorem}
Let $(M,g)$ be a compact Riemannian manifold of dimension $n\geq 6$ and $a$, 
$b$, $f$ \ smooth functions on $M$, $\lambda \in \left( 0,\lambda _{\ast
}\right) $, $1<q<2$ \ \ such that\newline
1) $f(x)>0$ on $M$.\newline
2) At the point $x_{\circ }$ where $f$ attains its maximum, we suppose, for $%
n=6$ 
\begin{equation*}
S_{g}(x_{\circ })+3a(x_{\circ })>0
\end{equation*}%
and for $n>6$ 
\begin{equation*}
\left( \frac{\left( n^{2}+4n-20\right) }{2(n+2)(n-6)}S_{g}(x_{\circ })+\frac{%
(n-1)}{(n+2)(n-6)}a(x_{\circ })-\frac{1}{8}\frac{\Delta f(x_{\circ })}{%
f(x_{\circ })}\right) >0\text{ .}
\end{equation*}%
Then the equation 
\begin{equation*}
\Delta _{g}^{2}u+div_{g}\left( a(x)\nabla _{g}u\right) +b(x)u=\lambda
\left\vert u\right\vert ^{q-2}u+f(x)\left\vert u\right\vert ^{N-2}u
\end{equation*}%
admits a non trivial solution of class $C^{4,\alpha }\left( M\right) $, $%
\alpha \in (0,1)$.
\end{theorem}

Recently, F. Madani \cite{14}, has considered the Yamabe problem with
singularities which he solved under some geometric conditions. The first
author in \cite{7} considered fourth order elliptic equation with
singularities of the form%
\begin{equation}
\Delta ^{2}u-\nabla ^{i}\left( a(x)\nabla _{i}u\right) +b(x)u=f\left\vert
u\right\vert ^{N-2}u  \tag{1}  \label{1}
\end{equation}%
where the functions $a$ and $b$ are in $L^{s}(M)$, $s>\frac{n}{2}$ and in $%
L^{p}(M)$, $p>\frac{n}{4}$ respectively, $N=\frac{2n}{n-4}$ is the Sobolev
critical exponent in the embedding $H_{2}^{2}\left( M\right) \hookrightarrow
L^{N}\left( M\right) $. He established the following result. Let $\left(
M,g\right) $ be a compact $n$-dimensional Riemannian manifold, $n\geq 6$, $%
a\in L^{s}(M)$, $b\in L^{p}(M)$, with $s>\frac{n}{2}$, $p>\frac{n}{4}$, $f$ $%
\in C^{\infty }(M)$ a positive function and $x_{o}\in M$ such that $%
f(x_{o})=\max_{x\in M}f(x)$.

\begin{theorem}
For $n\geq 10$, or $n=8,9$ and $2<p<5$, $\frac{9}{4}<s<11$or $n=7$, $\frac{7%
}{2}<s<9$ and $\frac{7}{4}<p<9$ we suppose that 
\begin{equation*}
\frac{n^{2}+4n-20}{6\left( n-6\right) (n^{2}-4)}S_{g}\left( x_{o}\right) -%
\frac{n-4}{2n\left( n-2\right) }\frac{\Delta f(x_{o})}{f(x_{o})}>0\text{.}
\end{equation*}%
For $n=6$ and $\frac{3}{2}<p<2$, $3<s<4$, we suppose that 
\begin{equation*}
S_{g}(x_{o})>0\text{.}
\end{equation*}%
Then the equation (\ref{1}) has a non trivial weak solution $u$ in $%
H_{2}^{2}\left( M\right) $. Moreover if $a\in H_{1}^{s}\left( M\right) $,
then $u\in $ $C^{0,\beta }\left( M\right) $, for some $\beta \in \left( 0,1-%
\frac{n}{4p}\right) $.
\end{theorem}

In this paper, we are concerned with the following problem: let $(M,g)$ be a
Riemannian compact manifold of dimension $n\geq 5$. Let $a\in L^{r}(M)$, $%
b\in L^{s}(M)$ where \ $r>\frac{n}{2}$, $s>\frac{n}{4}$ and $f$ \ a positive 
$C^{\infty }$-function on $M$; we look for non trivial solution of the
equation 
\begin{equation}
\Delta _{g}^{2}u+div_{g}\left( a(x)\nabla _{g}u\right) +b(x)u=\lambda
\left\vert u\right\vert ^{q-2}u+f(x)\left\vert u\right\vert ^{N-2}u  \tag{2}
\label{2}
\end{equation}%
where $1<q<2$ and $N=\frac{2n}{n-4}$ is the critical Sobolev exponent and $%
\lambda >0$ a real number. Our main result states as follows

\begin{theorem}
Let\ $(M,g)$ be a compact Riemannian manifold of dimension $n\geq 6$ and $f$
a positive function. Suppose that $P_{g}$ is coercive and at a point $x_{o}$
where $f$ attains its maximum the following conditions

\begin{equation}
\left\{ 
\begin{array}{c}
\frac{\Delta f(x_{o})}{f\left( x_{o}\right) }<\left( \frac{n\left(
n^{2}+4n-20\right) }{3\left( n+2\right) \left( n-4\right) \left( n-6\right) }%
\frac{1}{\left( 1+\left\Vert a\right\Vert _{r}+\left\Vert b\right\Vert
_{s}\right) ^{\frac{4}{n}}}-\frac{n-2}{3\left( n-1\right) }\right)
S_{g}\left( x_{o}\right) \text{ in case }n>6 \\ 
S_{g}(x_{\circ })>0\text{ \ in case }n=6\text{.}%
\end{array}%
\right.  \tag{C}  \label{C}
\end{equation}%
are true.

Then there is $\lambda ^{\ast }>0$ such that for any $\lambda \in (0,$ $%
\lambda ^{\ast })$, the equation (\ref{2}) has a non trivial weak solution.
\end{theorem}

For fixed $R\in M$, we define the function $\rho $ on $M$ by

\begin{equation}
\rho (Q)=\left\{ 
\begin{array}{c}
d(R,Q)\text{ \ if \ \ \ \ \ \ }d(R,Q)<\delta (M) \\ 
\delta (M)\text{ \ if\ \ \ \ \ }d(R,Q)\geq \delta (M)%
\end{array}%
\right.  \tag{3}  \label{3}
\end{equation}%
where $\delta (M)$ denotes the injectivity radius of $M$.

For real numbers $\sigma $ and $\mu $, consider the equation in the
distribution sense

\begin{equation}
\Delta ^{2}u-\nabla ^{i}(\frac{a}{\rho ^{\sigma }}\nabla _{i}u)+\frac{bu}{%
\rho ^{\mu }}=\lambda \left\vert u\right\vert ^{q-2}u+f(x)\left\vert
u\right\vert ^{N-2}u  \tag{4}  \label{4}
\end{equation}%
where the functions $a$ and $b$ are smooth on $M$,

\begin{corollary}
Let $0<\sigma <\frac{n}{r}<2$ and $0<\mu <\frac{n}{s}<4$. Suppose that%
\begin{equation*}
\left\{ 
\begin{array}{c}
\frac{\Delta f(x_{o})}{f\left( x_{o}\right) }<\frac{1}{3}\left( \frac{%
(n-1)n\left( n^{2}+4n-20\right) }{\left( n^{2}-4\right) \left( n-4\right)
\left( n-6\right) }\frac{1}{\left( 1+\left\Vert a\right\Vert _{r}+\left\Vert
b\right\Vert _{s}\right) ^{\frac{4}{n}}}-1\right) S_{g}\left( x_{o}\right) 
\text{ in case }n>6 \\ 
S_{g}(x_{\circ })>0\text{ \ in case }n=6\text{.}%
\end{array}%
\right.
\end{equation*}

Then there is $\lambda _{\ast }>0$ such that if $\lambda \in (0,$ $\lambda
_{\ast })$, the equation (\ref{4}) possesses a weak non trivial solution $%
u_{\sigma ,\mu }\in M_{\lambda }$.
\end{corollary}

In the sharp case $\sigma =2$ and $\mu =4$, letting $K(n,2,\gamma )$ the
best constant in the Hardy-Sobolev inequality given by Theorem \ref{th7} we
obtain the following result

\begin{theorem}
Let $(M,g)$ be a Riemannian compact manifold of dimension $n\geq 5$. Let $%
\left( u_{_{\sigma _{m},\mu _{m}}}\right) _{m}$ be a sequence in $M_{\lambda
}$ such that 
\begin{equation*}
\left\{ 
\begin{array}{c}
J_{\lambda ,\sigma ,\mu }(u_{_{\sigma _{m},\mu _{m}}})\leq c_{\sigma ,\mu }
\\ 
\nabla J_{\lambda }(u_{_{\sigma ,\mu }})-\mu _{_{\sigma ,\mu }}\nabla \Phi
_{\lambda }(u_{_{\sigma ,\mu }})\rightarrow 0%
\end{array}%
\right. \text{.}
\end{equation*}%
Suppose that 
\begin{equation*}
c_{\sigma ,\mu }<\frac{2}{n\text{ }K_{o}^{\frac{n}{4}}(f(x_{\circ }))^{\frac{%
n-4}{4}}}
\end{equation*}%
and 
\begin{equation*}
1+a^{-}\max \left( K(n,2,\sigma ),A\left( \varepsilon ,\sigma \right)
\right) +b^{-}\max \left( K(n,2,\mu ),A\left( \varepsilon ,\mu \right)
\right) >0
\end{equation*}%
then the equation%
\begin{equation*}
\Delta ^{2}u-\nabla ^{\mu }(\frac{a}{\rho ^{2}}\nabla _{\mu }u)+\frac{bu}{%
\rho ^{4}}=f\left\vert u\right\vert ^{N-2}u+\lambda \left\vert u\right\vert
^{q-2}u
\end{equation*}%
in the distribution has a weak non trivial solution.
\end{theorem}

\section{Existence of solutions}

In this section we focus on the existence of solutions to equation (\ref{1}%
); we use a variational method so we consider on $H_{2}^{2}(M)$ the
functional 
\begin{equation*}
J_{\lambda }(u)=\frac{1}{2}\int_{M}\left( \left\vert \Delta _{g}u\right\vert
^{2}-a(x)\left\vert \nabla _{g}u\right\vert ^{2}+b(x)u^{2}\right) dv_{g}-%
\frac{\lambda }{q}\int_{M}\left\vert u\right\vert ^{q}dv_{g}-\frac{1}{N}%
\int_{M}f(x)\left\vert u\right\vert ^{N}dv_{g}\text{.}
\end{equation*}%
First, we put%
\begin{equation*}
\Phi _{\lambda }(u)=\left\langle \nabla J_{\lambda }(u),\text{ }%
u\right\rangle
\end{equation*}%
hence%
\begin{equation*}
\Phi _{\lambda }(u)=\int_{M}\left( \left( \Delta _{g}u\right)
^{2}-a(x)\left\vert \nabla _{g}u\right\vert ^{2}+b(x)u^{2}\right)
dv_{g}-\lambda \int_{M}\left\vert u\right\vert
^{q}dv_{g}-\int_{M}f(x)\left\vert u\right\vert ^{N}dv_{g}\text{.}
\end{equation*}%
We let%
\begin{equation*}
M_{\lambda }=\left\{ u\in H_{2}^{2}(M):\text{ }\Phi _{\lambda }(u)=0\text{
and }\left\Vert u\right\Vert \geq \tau >0\right\} \text{.}
\end{equation*}%
The operator $P_{g}(u)$ is said coercive if there exits $\Lambda >0$ such
that for any $u\in H_{2}^{2}(M)$ 
\begin{equation*}
\int_{M}uP_{g}(u)dv_{g}\geq \Lambda \left\Vert u\right\Vert
_{H_{2}^{2}(M)}^{2}.
\end{equation*}

\begin{proposition}
$\left\Vert u\right\Vert =(\int_{M}\left\vert \Delta _{g}u\right\vert
^{2}-a(x)\left\vert \nabla _{g}u\right\vert ^{2}+b(x)u^{2}dv_{g})^{\frac{1}{2%
}}$ is an equivalent norm to the usual one on $H_{2}^{2}(M)$ if and only if $%
P_{g}$ is coercive.
\end{proposition}

\begin{proof}
If $P_{g}$ is coercive there is $\Lambda >0$ such that for any $u\in
H_{2}^{2}(M)$, 
\begin{equation*}
\int_{M}P_{g}(u)udv_{g}\geq \Lambda \left\Vert u\right\Vert
_{H_{2}^{2}(M)}^{2}
\end{equation*}%
and since $a\in L^{r}(M)$ and $b\in L^{s}(M)$ where\ $r>\frac{n}{2}$ and $s>%
\frac{n}{4}$, by H\"{o}lder's inequality we get%
\begin{equation*}
\int_{M}uP_{g}(u)dv_{g}\leq \left\Vert \Delta _{g}u\right\Vert
_{2}^{2}+\left\Vert a\right\Vert _{\frac{n}{2}}\left\Vert \nabla
_{g}u\right\Vert _{2^{\ast }}^{2}+\left\Vert b\right\Vert _{\frac{n}{4}%
}\left\Vert u\right\Vert _{N}^{2}
\end{equation*}%
where $2^{\ast }=\frac{2n}{n-2}$.

The Sobolev's inequalities lead to : for any $\eta >0$ 
\begin{equation*}
\left\Vert \nabla _{g}u\right\Vert _{2^{\ast }}^{2}\leq \max ((1+\eta
)K(n,1)^{2},A_{\eta })\int_{M}\left( \left\vert \nabla _{g}^{2}u\right\vert
^{2}+\left\vert \nabla _{g}u\right\vert ^{2}\right) dv_{g}
\end{equation*}%
where $K(n,1)$ denotes the best Sobolev's constant in the embedding $%
H_{1}^{2}\left( R^{n}\right) \hookrightarrow L^{\frac{2n}{n-2}}\left(
R^{n}\right) $, and for any $\epsilon >0$%
\begin{equation*}
\left\Vert u\right\Vert _{N}^{2}\leq \max ((1+\varepsilon
)K_{o},B_{\varepsilon })\left\Vert u\right\Vert _{H_{2}^{2}(M)}^{2}
\end{equation*}%
where in this latter inequality $K_{o}$ is the best Sobolev's constant in
the embedding $H_{1}^{2}\left( M\right) \hookrightarrow L^{\frac{2n}{n-2}%
}\left( M\right) $ and $B_{\epsilon }$ the corresponding see ( see \cite{3}%
). Now by the well known formula (see \cite{3}, page 115)%
\begin{equation*}
\int_{M}\left\vert \nabla _{g}^{2}u\right\vert ^{2}dv_{g}=\int_{M}\left(
\left\vert \Delta _{g}u\right\vert ^{2}-R_{ij}\nabla ^{i}u\nabla
^{j}u\right) dv_{g}
\end{equation*}%
where $R_{ij}$ denote the components of the Ricci curvature, there is a
constant $\beta >0$ such that%
\begin{equation*}
\int_{M}\left\vert \nabla _{g}^{2}u\right\vert ^{2}dv_{g}\leq
\int_{M}\left\vert \Delta _{g}u\right\vert ^{2}+\beta \left\vert \nabla
_{g}u\right\vert ^{2}dv_{g}
\end{equation*}%
so we get%
\begin{equation*}
\left\Vert \nabla _{g}u\right\Vert _{2^{\ast }}^{2}\leq (\beta +1)\max
((1+\eta )K(n,1)^{2},A_{\eta })\int_{M}\left( \left\vert \Delta
_{g}u\right\vert ^{2}+\left\vert \nabla _{g}u\right\vert ^{2}+u^{2}\right)
dv_{g}
\end{equation*}%
and we infer that%
\begin{equation*}
\int_{M}P_{g}(u)udv_{g}\leq \left\Vert u\right\Vert
_{H_{2}^{2}(M)}^{2}+(\beta +1)\left\Vert a\right\Vert _{\frac{n}{2}}\max
((1+\eta )K(n,1)^{2},A_{\eta })\left\Vert u\right\Vert _{H_{2}^{2}(M)}^{2}+
\end{equation*}%
\begin{equation*}
\left\Vert b\right\Vert _{\frac{n}{4}}\max ((1+\varepsilon
)K_{o},B_{\varepsilon })\left\Vert u\right\Vert _{H_{2}^{2}(M)}^{2}\text{.}
\end{equation*}%
Hence%
\begin{equation*}
\int_{M}uP_{g}(u)dv_{g}\leq \underbrace{\max \left( 1,\left\Vert
b\right\Vert _{\frac{n}{4}}\max ((1+\varepsilon )K_{o},B_{\varepsilon
}),(\beta +1)\left\Vert a\right\Vert _{\frac{n}{2}}\max ((1+\varepsilon
)K(n,1)^{2},A_{\varepsilon })\right) }_{>0}\left\Vert u\right\Vert
_{H_{2}^{2}(M)}^{2}\text{.}
\end{equation*}%
.
\end{proof}

\begin{lemma}
\label{lem} The set $M_{\lambda }$ is non empty provided that $\lambda $ $%
\in (0,$ $\lambda _{\circ })$ where 
\begin{equation*}
\lambda _{\circ }=\frac{\left( 2^{q-2}-2^{q-N}\right) \Lambda ^{\frac{N-q}{%
N-2}}}{V(M)^{(1-\frac{q}{N})}\left( \max_{x\in M}f(x)\right) ^{\frac{2-q}{N-2%
}}(\max ((1+\varepsilon )K\left( n,2\right) ,A_{\varepsilon }))^{\frac{N-q}{%
N-2}}}\text{.}
\end{equation*}
\end{lemma}

\begin{proof}
The proof of Lemma \ref{lem} is the same as in (\cite{8}), but we give it
here for convenience. Let $t>0$ and $u\in H_{2}^{2}(M)-\left\{ 0\right\} $.
Evaluating $\Phi _{\lambda }$ at $tu$, we get 
\begin{equation*}
\Phi _{\lambda }(tu)=t^{2}\left\Vert u\right\Vert ^{2}-\lambda
t^{q}\left\Vert u\right\Vert _{q}^{q}-t^{N}\dint\limits_{M}f(x)\left\vert
u\right\vert ^{N}dv_{g}\text{.}
\end{equation*}%
\newline
Put%
\begin{equation*}
\alpha (t)=\left\Vert u\right\Vert
^{2}-t^{N-2}\dint\limits_{M}f(x)\left\vert u\right\vert ^{N}dv(g)
\end{equation*}%
and%
\begin{equation*}
\beta (t)=\lambda t^{q-2}\left\Vert u\right\Vert _{q}^{q}\text{;}
\end{equation*}%
by Sobolev's inequality, we get%
\begin{equation*}
\alpha (t)\geq \left\Vert u\right\Vert ^{2}-\max_{x\in M}f(x)(\max
((1+\varepsilon )K_{\circ },A_{\varepsilon }))^{\frac{N}{2}}\left\Vert
u\right\Vert _{H_{2}^{2}(M)}^{N}t^{N-2}\text{.}
\end{equation*}%
By the coercivity of the operator $P_{g}=\Delta _{g}^{2}-div_{g}\left(
a\nabla _{g}\right) +b$ there is a constant $\Lambda >0$ such that 
\begin{equation*}
\alpha (t)\geq \left\Vert u\right\Vert ^{2}-\Lambda ^{-\frac{N}{2}%
}\max_{x\in M}f(x)(\max ((1+\varepsilon )K_{\circ },A_{\varepsilon }))^{%
\frac{N}{2}}\left\Vert u\right\Vert ^{N}t^{N-2}.
\end{equation*}%
Letting%
\begin{equation*}
\alpha _{1}(t)=\left\Vert u\right\Vert ^{2}-\Lambda ^{-\frac{N}{2}%
}\max_{x\in M}f(x)(\max ((1+\varepsilon )K_{\circ },A_{\varepsilon }))^{%
\frac{N}{2}}\left\Vert u\right\Vert ^{N}t^{N-2}
\end{equation*}%
H\"{o}lder and Sobolev inequalities lead to%
\begin{equation*}
\beta (t)\leq \lambda V(M)^{(1-\frac{q}{N})}(\max ((1+\varepsilon )K_{\circ
},A_{\varepsilon }))^{\frac{q}{2}}\left\Vert u\right\Vert
_{H_{2}^{2}(M)}^{q}t^{q-2}
\end{equation*}%
and the coercivity of $P_{g}$ assures the existence of a constant $\Lambda
>0 $ such that 
\begin{equation*}
\beta (t)\leq \lambda \Lambda ^{-\frac{q}{2}}V(M)^{(1-\frac{q}{N})}(\max
((1+\varepsilon )K_{\circ },A_{\varepsilon }))^{\frac{q}{2}}\left\Vert
u\right\Vert ^{q}t^{q-2}\text{.}
\end{equation*}%
Put%
\begin{equation*}
\beta _{1}(t)=\lambda \Lambda ^{-\frac{q}{2}}V(M)^{(1-\frac{q}{N})}(\max
((1+\varepsilon )K_{\circ },A_{\varepsilon }))^{\frac{q}{2}}\left\Vert
u\right\Vert ^{q}t^{q-2}\text{.}
\end{equation*}%
Let $t_{o}$ such $\alpha _{1}(t_{o})=0$ \ i.e. \ 
\begin{equation*}
t_{\circ \text{ }}=\frac{\Lambda ^{\frac{N}{2(N-2)}}}{\left\Vert
u\right\Vert \left( \max_{x\in M}f(x)\right) ^{\frac{1}{N-2}}(\max
((1+\varepsilon )K_{\circ },A_{\varepsilon }))^{\frac{N}{2(N-2)}}}
\end{equation*}%
Now since $\alpha _{1}(t)$ is a decreasing and a concave function and $\beta
_{1}(t)$ is a decreasing and convex function, \ then%
\begin{equation*}
\min_{t\text{ }\in (0,\text{ }\frac{t_{\circ }}{2}]}\alpha _{1}(t)=\alpha
_{1}(\frac{t_{\circ }}{2})=\left\Vert u\right\Vert ^{2}(1-2^{2-N})>0
\end{equation*}%
and 
\begin{equation*}
\min_{t\text{ }\in (0,\text{ }\frac{t_{\circ }}{2}]}\beta _{1}(t)=\beta _{1}(%
\frac{t_{\circ }}{2})>0
\end{equation*}%
where

\begin{equation*}
\beta _{1}(\frac{t_{\circ }}{2})=\frac{2^{2-q}\lambda V(M)^{(1-\frac{q}{N}%
)}\Lambda ^{\frac{q-N}{N-2}}\left\Vert u\right\Vert ^{2}}{(\max
((1+\varepsilon )K_{\circ },A_{\varepsilon }))^{\frac{q-N}{N-2}}\left(
\max_{x\in M}f(x)\right) ^{\frac{q-2}{N-2}}}\text{.}
\end{equation*}%
Consequently $\Phi _{\lambda }(tu)=0$ with $t$ $\in (0,$ $\frac{t_{\circ }}{2%
}]$ has a solution if 
\begin{equation*}
\min_{t\text{ }\in (0,\text{ }\frac{t_{\circ }}{2}]}\alpha _{1}(t)\geq
\max_{t\text{ }\in (0,\text{ }\frac{t_{\circ }}{2}]}\beta _{1}(t)
\end{equation*}%
that is to say 
\begin{equation*}
\ 0<\lambda <\frac{\left( 2^{q-2}-2^{q-N}\right) \left( \max_{x\in
M}f(x)\right) ^{\frac{q-2}{N-2}}(\max ((1+\varepsilon )K_{\circ
},A_{\varepsilon }))^{\frac{q-N}{N-2}}}{\Lambda ^{\frac{N-q}{N-2}}V(M)^{(1-%
\frac{q}{N})}}=\lambda _{\circ }
\end{equation*}%
Let $t_{1}\in (0,$ $\frac{t_{\circ }}{2}]$ such that $\Phi _{\lambda
}(t_{1}u)=0$. If we take $u\in H_{2}^{2}\left( M\right) $ such that $%
\left\Vert u\right\Vert \geq \frac{\rho }{t_{1}}$ and $v=t_{1}u$ we get $%
\Phi _{\lambda }(v)=0$ and $\left\Vert v\right\Vert =t_{1}\left\Vert
u\right\Vert \geq \rho $ i.e. $v\in M_{\lambda }$ provided that $\lambda $ $%
\in (0,$ $\lambda _{\circ })$.
\end{proof}

\section{Geometric conditions of $J_{\protect\lambda }$}

The following lemmas whose proofs are the same as in \cite{8} will be useful.

\begin{lemma}
\label{lem1} Let $(M,g)$ be a Riemannian compact manifold of dimension $%
n\geq 5$. For all $u\in M_{\lambda }$ and all $\lambda \in \left( 0,\min
\left( \lambda _{\circ },\lambda _{1}\right) \right) $ there is $A$ $>$ $0$
such that $J_{\lambda }(u)\geq A>0$ where 
\begin{equation*}
\lambda _{1}=\frac{\frac{\left( N-2\right) q}{2\left( N-q\right) }\Lambda ^{%
\frac{q}{2}}}{V(M)^{1-\frac{q}{N}}(\max ((1+\varepsilon
)K(n,2),A_{\varepsilon }))^{\frac{q}{2}}\tau ^{q-2}}\text{.}
\end{equation*}
\end{lemma}

\begin{lemma}
\label{lem2} Let $(M,g)$ be a Riemannian compact manifold of dimension $%
n\geq 5$. The following assertions are true:

(i)$\ \left\langle \nabla \Phi _{\lambda }(u),u\right\rangle <0$ for all $%
u\in M_{\lambda }$ and for all $\lambda \in (0,\min (\lambda _{\circ
},\lambda _{1})).$

(ii)\ The critical points of $J_{\lambda }$ are points of $M_{\lambda }.$
\end{lemma}

\section{Existence of non trivial solution in $M_{\protect\lambda }$}

In this section, first we show that $J_{\lambda }$ satisfies the
Palais-Smale condition on $M_{\lambda }$ provided that $\lambda >0$ is
sufficiently small.

\begin{lemma}
\label{lem3} Let $(M,g)$ be a compact Riemannian manifold of dimension $%
n\geq 5$. Let $\left( u_{m}\right) _{m}$ be a sequence in $M_{\lambda }$
such that 
\begin{equation*}
\left\{ 
\begin{array}{c}
J_{\lambda }(u_{m})\leq c \\ 
\nabla J_{\lambda }(u_{m})-\mu _{m}\nabla \Phi _{\lambda }(u_{m})\rightarrow
0%
\end{array}%
\text{.}\right.
\end{equation*}%
Suppose that 
\begin{equation*}
c<\frac{2}{n\text{ }K_{o}^{\frac{n}{4}}(f(x_{\circ }))^{\frac{n-4}{4}}}
\end{equation*}%
then there is a subsequence $\left( u_{m}\right) _{m}$ converging strongly
in $H_{2}^{2}(M)$.
\end{lemma}

\begin{proof}
Let $\left( u_{m}\right) _{m}\subset M_{\lambda }$%
\begin{equation*}
J_{\lambda }(u_{m})=\frac{N-2}{2N}\left\Vert u_{m}\right\Vert ^{2}-\lambda 
\frac{N-q}{Nq}\int_{M}\left\vert u_{m}\right\vert ^{q}dv_{g}
\end{equation*}%
As in the proof of Lemma \ref{2}, we have \ 
\begin{equation*}
J_{\lambda }(u_{m})\geq \frac{N-2}{2N}\left\Vert u_{m}\right\Vert
^{2}-\lambda \frac{N-q}{Nq}\Lambda ^{-\frac{q}{2}}V(M)^{1-\frac{q}{N}}(\max
((1+\varepsilon )K_{o},A_{\varepsilon }))^{\frac{q}{2}}\left\Vert
u_{m}\right\Vert ^{q}
\end{equation*}%
\begin{equation*}
J_{\lambda }(u_{m})\geq \left\Vert u_{m}\right\Vert ^{2}(\frac{N-2}{2N}%
-\lambda \frac{N-q}{Nq}\Lambda ^{-\frac{q}{2}}V(M)^{1-\frac{q}{N}}(\max
((1+\varepsilon )K_{o},A_{\varepsilon }))^{\frac{q}{2}}\tau ^{q-2})>0
\end{equation*}%
and since $0<\lambda <\frac{\frac{\left( N-2\right) q}{2\left( N-q\right) }%
\Lambda ^{\frac{q}{2}}}{V(M)^{1-\frac{q}{N}}(\max ((1+\varepsilon
)K(n,2),A_{\varepsilon }))^{\frac{q}{2}}\tau ^{q-2}}$ and $J_{\lambda
}(u_{m})\leq c$, we get 
\begin{equation*}
c\geq J_{\lambda }(u_{m})
\end{equation*}%
\begin{equation*}
\geq \lbrack \frac{N-2}{2N}-\lambda \frac{N-q}{Nq}\Lambda ^{-\frac{q}{2}%
}V(M)^{1-\frac{q}{N}}(\max ((1+\varepsilon )K_{o},A_{\varepsilon }))^{\frac{q%
}{2}}\tau ^{q-2}]\left\Vert u_{m}\right\Vert ^{2}>0
\end{equation*}%
so 
\begin{equation*}
\left\Vert u_{m}\right\Vert ^{2}\leq \frac{c}{\frac{N-2}{2N}-\lambda \frac{%
N-q}{Nq}\Lambda ^{-\frac{q}{2}}V(M)^{1-\frac{q}{N}}(\max ((1+\varepsilon
)K_{o},A_{\varepsilon }))^{\frac{q}{2}}\tau ^{q-2}}<+\infty \text{.}
\end{equation*}%
$\left( u_{m}\right) _{m}$ is a bounded in $H_{2}^{2}(M)$.\ By the
compactness of the embedding $H_{2}^{2}(M)\subset H_{p}^{k}(M)$\ ($k\ =0,1$; 
$p<N$) we get a subsequence still denoted $\left( u_{m}\right) _{m}$ such
that

$u_{m}\rightarrow u$\ \ weakly in $H_{2}^{2}(M)$

$u_{m}\rightarrow u$ strongly in $L^{p}(M)$\ where $p<N$

$\nabla u_{m}\rightarrow \nabla u$ strongly in $L^{p}(M)$\ where $p<2^{\ast
}=\frac{2n}{n-2}$

$u_{m}\rightarrow u$ a.e. in $M.$\newline
On the other hand since $\frac{2s}{s-1}<N=\frac{2n}{n-4}$, we obtain 
\begin{equation*}
\left\vert \int_{M}b(x)\left\vert u_{m}-u\right\vert ^{2}dv_{g}\right\vert
\leq \left\Vert b\right\Vert _{s}\left\Vert u_{m}-u\right\Vert _{\frac{2s}{%
s-1}}^{2}
\end{equation*}%
\begin{equation*}
\leq \left\Vert b\right\Vert _{s}\left( \left( K_{\circ }+\epsilon \right)
\left\Vert \Delta \left( u_{m}-u\right) \right\Vert _{2}^{2}+A_{\epsilon
}\left\Vert u_{m}-u\right\Vert _{2}^{2}\right) \text{.}
\end{equation*}%
Now taking account of 
\begin{equation}
K_{\circ }=\frac{16}{n(n^{2}-4)\left( n-4\right) \omega _{n}^{\frac{n}{4}}}<1
\tag{K}  \label{K}
\end{equation}%
we get%
\begin{equation*}
\int_{M}b(x)\left( u_{m}-u\right) ^{2}dv_{g}\leq \left\Vert b\right\Vert
_{s}\left\Vert \Delta \left( u_{m}-u\right) \right\Vert _{2}^{2}+o\left(
1\right) \text{.}
\end{equation*}

\ By the same procedure as above we get%
\begin{equation*}
\int_{M}a(x)\left\vert \nabla \left( u_{m}-u\right) \right\vert
^{2}dv_{g}\leq \left\Vert a\right\Vert _{r}\left\Vert \Delta \left(
u_{m}-u\right) \right\Vert _{2}^{2}+o\left( 1\right) \text{.}
\end{equation*}%
By Brezis-Lieb lemma we write%
\begin{equation*}
\int_{M}\left( \Delta _{g}u_{m}\right) ^{2}dv_{g}=\int_{M}\left( \Delta
_{g}u\right) ^{2}dv_{g}+\int_{M}\left( \Delta _{g}(u_{m}-u)\right)
^{2}dv_{g}+o(1)
\end{equation*}%
and also 
\begin{equation*}
\int_{M}f(x)\left\vert u_{m}\right\vert ^{N}dv_{g}=\int_{M}f(x)\left\vert
u\right\vert ^{N}dv_{g}+\int_{M}f(x)\left\vert u_{m}-u\right\vert
^{N}dv_{g}+o(1)\text{.}
\end{equation*}%
Now we claim that $\mu _{m}\rightarrow 0$ as $m$ $\rightarrow +\infty $%
\newline
Testing with $u_{m}$ we obtain 
\begin{equation*}
\left\langle \nabla J_{\lambda }(u_{m})-\mu _{m}\nabla \Phi _{\lambda
}(u_{m}),u_{m}\right\rangle =o(1)
\end{equation*}%
\begin{equation*}
=\underset{=0}{\underbrace{\left\langle \nabla J_{\lambda
}(u_{m}),u_{m}\right\rangle }}-\mu _{m}\left\langle \nabla \Phi _{\lambda
}(u_{m}),u_{m}\right\rangle =o(1)
\end{equation*}%
hence 
\begin{equation*}
\mu _{m}\left\langle \nabla \Phi _{\lambda }(u_{m}),u_{m}\right\rangle =o(1)%
\text{.}
\end{equation*}%
By Lemma \ref{lem2}, we get $\lim \sup_{m\newline
}\left\langle \nabla \Phi _{\lambda }(u_{m}),u_{m}\right\rangle <0$ so 
\begin{equation*}
\mu _{m}\rightarrow 0\text{ as\ }m\rightarrow +\infty \text{.}
\end{equation*}%
Our last claim is that $u_{m}\rightarrow u$ strongly in $H_{2}^{2}(M)$,
indeed 
\begin{equation*}
J_{\lambda }(u_{m})-J_{\lambda }(u)
\end{equation*}%
\begin{equation*}
=\frac{1}{2}\int_{M}\left( \Delta _{g}(u_{m}-u)\right) ^{2}dv_{c}-\frac{1}{N}%
\int_{M}f(x)\left\vert u_{m}-u\right\vert ^{N}dv_{g}+o(1)\text{.}
\end{equation*}%
Since $u_{m}-u\rightarrow 0$\ weakly in $H_{2}^{2}(M)$, we test with $\nabla
J_{\lambda }(u_{m})-\nabla J_{\lambda }(u)$ 
\begin{equation*}
\left\langle \nabla J_{\lambda }(u_{m})-\nabla J_{\lambda
}(u),u_{m}-u\right\rangle =o(1)
\end{equation*}%
\begin{equation}
=\int_{M}\left( \Delta _{g}(u_{m}-u)\right)
^{2}dv_{g}-\int_{M}f(x)\left\vert u_{m}-u\right\vert ^{N}dv_{g}=o(1)  \tag{5}
\label{5}
\end{equation}%
and get 
\begin{equation*}
\int_{M}\left( \Delta _{g}(u_{m}-u)\right) ^{2}dv_{g}=\int_{M}f(x)\left\vert
u_{m}-u\right\vert ^{N}dv_{g}+o(1)
\end{equation*}%
and taking account of (\ref{5}) we obtain%
\begin{equation*}
J_{\lambda }(u_{m})-J_{\lambda }(u)=\frac{1}{2}\int_{M}\left( \Delta
_{g}(u_{m}-u)\right) ^{2}dv_{g}-\frac{1}{N}\int_{M}\left( \Delta
_{g}(u_{m}-u)\right) ^{2}dv_{g}+o(1)
\end{equation*}%
$\ $i.e.%
\begin{equation*}
J_{\lambda }(u_{m})-J_{\lambda }(u)=\frac{2}{n}\int_{M}\left( \Delta
_{g}(u_{m}-u)\right) ^{2}dv_{g}+o(1)\text{.}
\end{equation*}%
Independently, by the Sobolev's inequality we have 
\begin{equation}
\left\Vert u_{m}-u\right\Vert _{N}^{2}\leq (1+\varepsilon
)K_{o}\int_{M}\left( \Delta _{g}(u_{m}-u)\right) ^{2}dv_{g}+o(1)\text{.} 
\tag{6}  \label{6}
\end{equation}%
Since%
\begin{equation*}
\int_{M}f(x)\left\vert u_{m}-u\right\vert ^{N}dv_{g}\leq \max_{x\in
M}f(x)\left\Vert u_{m}-u\right\Vert _{N}^{N}
\end{equation*}%
we infer by (\ref{6}) that%
\begin{equation*}
\int_{M}f(x)\left\vert u_{m}-u\right\vert ^{N}dv_{g}\leq (1+\varepsilon )^{%
\frac{n}{n-4}}\max_{x\in M}f(x)K_{\circ }^{\frac{n}{n-4}}\left\Vert \Delta
_{g}(u_{m}-u)\right\Vert _{2}^{N}+o(1)
\end{equation*}%
and appealing equality (\ref{5})%
\begin{equation*}
o(1)\geq \left\Vert \Delta _{g}(u_{m}-u)\right\Vert _{2}^{2}-(1+\varepsilon
)^{\frac{n}{n-4}}\max_{x\in M}f(x)K_{o}^{\frac{n}{n-4}}\left\Vert \Delta
_{g}(u_{m}-u)\right\Vert _{2}^{N}+o(1)
\end{equation*}%
\begin{equation*}
\geq \left\Vert \Delta _{g}(u_{m}-u)\right\Vert _{2}^{2}(1-(1+\varepsilon )^{%
\frac{n}{n-4}}\max_{x\in M}f(x)K_{\circ }^{\frac{n}{n-4}}\left\Vert \Delta
_{g}(u_{m}-u)\right\Vert _{2}^{N-2})+o(1)
\end{equation*}%
so if \ 
\begin{equation}
\limsup_{\newline
m\rightarrow +\infty }\left\Vert \Delta _{g}(u_{m}-u)\right\Vert _{2}^{2}<%
\frac{1}{K_{\circ }^{\frac{n}{4}}\left( \max_{x\in M}f(x)\right) ^{\frac{n}{4%
}-1}}  \tag{8}  \label{8}
\end{equation}%
then $u_{m}\rightarrow u$ strongly in $H_{2}^{2}\left( M\right) $. The
condition (\ref{8}) is fulfilled since by Lemma \ref{lem1} $J_{\lambda
}(u)>0 $ on $M_{\lambda }$ with $\lambda $ is as in Lemma \ref{lem1} and by
hypothesis%
\begin{equation*}
c\geq J_{\lambda }\left( u_{m}\right) >\left( J_{\lambda }\left(
u_{m}\right) -J_{\lambda }\left( u\right) \right) =\frac{2}{n}\int_{M}\left(
\Delta _{g}(u_{m}-u)\right) ^{2}dv_{g}
\end{equation*}%
and 
\begin{equation*}
c<\ \frac{2}{n\text{ }K_{\circ }^{\frac{n}{4}}(\max_{x\in M}f(x))^{\frac{n}{4%
}-1}}.
\end{equation*}%
It is obvious that 
\begin{equation*}
\Phi _{\lambda }(u)=0\text{ \ and }\left\Vert u\right\Vert \geq \tau
\end{equation*}%
i.e. $u\in M_{\lambda }$.
\end{proof}

Now we show the existence of a sequence in $M_{\lambda }$ satisfying the
conditions of Palais-Smale.

\begin{lemma}
\label{lem4} Let $(M,g)$ be a compact Riemannian manifold of dimension $%
n\geq 5$, then there is a couple $\left( u_{m},\mu _{m}\right) \in
M_{\lambda }\times R$ such that $\nabla J_{\lambda }(u_{m})-\mu _{m}\nabla
\Phi _{\lambda }(u_{m})\rightarrow 0$ strongly in $(H_{2}^{2}(M))^{\ast }$
and $J_{\lambda }\left( u_{m}\right) $ is bounded provide that $\lambda \in
\left( 0,\lambda _{\ast }\right) $ with $\lambda _{\ast }=\left\{ \min
(\lambda _{\circ },\lambda _{1}),0\right\} $.
\end{lemma}

\begin{proof}
Since $J_{\lambda }$ is Gateau differentiable and by Lemma \ref{1} bounded
below on $M_{\lambda }$ it follows from Ekeland's principle \ that there is
a couple $(u_{m},$ $\mu _{m})$ $\in M_{\lambda }\times R$ such that $\nabla
J_{\lambda }(u_{m})-\mu _{m}\nabla \Phi _{\lambda }(u_{m})\rightarrow 0$
strongly in $(H_{2}^{2}(M))^{^{\prime }}$ and $J_{\lambda }\left(
u_{m}\right) $ is bounded i.e. $(u_{m},$ $\mu _{m})_{m}$ is a Palais-Smale
sequence on $M_{\lambda }$.
\end{proof}

Now we are in position to establish the following generic existence result.

\begin{theorem}
\label{thm3} Let\ $(M,g)$ be a compact Riemannian manifold of dimension $%
n\geq 5$ and $f$ a positive function. Suppose that $P_{g}$ is coercive and 
\begin{equation}
c<\frac{2}{n\text{ }K_{o}^{\frac{n}{4}}(f(x_{\circ }))^{\frac{n-4}{4}}}\text{%
.}  \tag{C1}  \label{C1}
\end{equation}%
Then there is $\lambda ^{\ast }>0$ such that for any $\lambda \in (0,$ $%
\lambda ^{\ast })$, the equation (\ref{2}) has a non trivial weak solution.
\end{theorem}

\begin{proof}
By Lemma \ref{lem3} and \ref{lem4} there is $u\in H_{2}^{2}\left( M\right) $
such that 
\begin{equation*}
J_{\lambda }(u)=\min_{\varphi \in M_{\lambda }}J_{\lambda }(\varphi )\text{.}
\end{equation*}%
$\newline
$By Lagrange multiplicative theorem there is a real number $\mu $ such that
for any $\varphi \in H_{2}^{2}\left( M\right) $%
\begin{equation}
\left\langle \nabla J_{\lambda }(u),\varphi \right\rangle =\mu \left\langle
\nabla \Phi _{\lambda }(u),\varphi \right\rangle  \tag{9}  \label{9}
\end{equation}%
and letting $\varphi =u$ in the equation (\ref{9}), we get 
\begin{equation*}
\Phi _{\lambda }(u)=\left\langle \nabla J_{\lambda }(u),u\right\rangle =\mu
\left\langle \nabla \Phi _{\lambda }(u),u\right\rangle \text{.}
\end{equation*}%
By Lemma \ref{lem2} we get that $\mu =0$ and by equation (\ref{9}), we infer
that for any $\varphi \in H_{2}^{2}\left( M\right) $ 
\begin{equation*}
\left\langle \nabla J_{\lambda }(u),\varphi \right\rangle =0
\end{equation*}%
hence $u$ is weak non trivial solution to equation (\ref{2}) and since by
Lemma \ref{2} critical points of $J_{\lambda }$, we conclude that $u\in
M_{\lambda }$.
\end{proof}

\section{Application}

Let $P\in M$, we define a function on $M$ by

\begin{equation}
\rho _{_{P}}(Q)=\left\{ 
\begin{array}{c}
d(P,Q)\text{ if }d(P,Q)<\delta (M) \\ 
\delta (M)\text{ if }d(P,Q)\geq \delta (M)%
\end{array}%
\right.  \tag{10}  \label{10}
\end{equation}%
where $\delta (M)$ is the injectivity radius of $M$. For brevity we denote
this function by $\rho $. The weighted $L^{p}\left( M,\rho ^{\gamma }\right) 
$ space will be the set of measurable functions $u$ on $M$ such that $\rho
^{\gamma }\left\vert u\right\vert ^{p}$ are integrable where $p\geq 1$. We
endow $L^{p}\left( M,\rho ^{\gamma }\right) $ with the norm

\begin{equation*}
\left\Vert u\right\Vert _{p,\rho }=\left( \int_{M}\rho ^{\gamma }\left\vert
u\right\vert ^{p}dv_{g}\right) ^{\frac{1}{p}}\text{.}
\end{equation*}%
In this section we will be in need of the following Hardy-Sobolev inequality
and Releich-Kondrakov embedding respectively whose proofs are given in (\cite%
{7}).

\begin{theorem}
\label{th7} Let $(M,g)$ be a Riemannian compact manifold of dimension $n\geq
5$ and $p$, $q$ , $\gamma $ are real numbers such that $\frac{\gamma }{p}=%
\frac{n}{q}-\frac{n}{p}-2$ $\ $and $2\leq p\leq \frac{2n}{n-4}$.$\newline
$For any $\ \epsilon >0$, there is $A(\epsilon ,q,\gamma )$ such that for
any $u\in H_{2}^{2}(M)$ 
\begin{equation}
\left\Vert u\right\Vert _{p,\rho ^{\gamma }}^{2}\leq (1+\epsilon
)K(n,2,\gamma )^{2}\left\Vert \Delta _{g}u\right\Vert _{2}^{2}+A(\epsilon
,q,\gamma )\left\Vert u\right\Vert _{2}^{2}  \tag{11}  \label{11}
\end{equation}%
where $K(n,2,\gamma )$ is the optimal constant.
\end{theorem}

In case $\gamma =0$, $K(n,2,0)=K(n,2)=K_{o}^{\frac{1}{2}}$ is the best
constant in the Sobolev's embedding of $H_{2}^{2}(M)$ in $L^{N}(M)$ where $N=%
\frac{2n}{n-4}$.

\begin{theorem}
Let $(M,g)$ be a compact Riemannian manifold of dimension $n\geq 5$ and $p$, 
$q$ , $\gamma $ are real numbers satisfying $1\leq q\leq p\leq \frac{nq}{n-2q%
}$ , $\gamma <0$ and $l=1$,$2$.

If $\frac{\gamma }{p}=n$ $(\frac{1}{q}-\frac{1}{p})-l$ then the inclusion $%
H_{l}^{q}(M)\subset $ $L^{p}(M,\rho ^{\gamma })$ is continuous.$\newline
$If $\frac{\gamma }{p}>n$ $(\frac{1}{q}-\frac{1}{p})-l$ then inclusion $%
H_{l}^{q}(M)\subset $ $L^{p}(M,\rho ^{\gamma })\ $is compact.
\end{theorem}

We consider the following equation%
\begin{equation}
\Delta _{g}^{2}u+div_{g}\left( \frac{a(x)}{\rho ^{\sigma }}\nabla
_{g}u\right) +\frac{b(x)}{\rho ^{\mu }}u=\lambda \left\vert u\right\vert
^{q-2}u+f(x)\left\vert u\right\vert ^{N-2}u  \tag{12}  \label{12}
\end{equation}%
where $a$ and $b$ are smooth function and $\rho $ denotes the distance
function defined by (\ref{10}), $\lambda >0$ in some interval $\left(
0,\lambda _{\ast }\right) $, $1<q<2$, $\sigma $, $\mu $ will be precise
later and we associate to (\ref{12}) on $H_{2}^{2}(M)$ the functional 
\begin{equation*}
J_{\lambda }(u)=\frac{1}{2}\int_{M}\left( \left( \Delta _{g}u\right) ^{2}-%
\frac{a(x)}{\rho ^{\sigma }}\left\vert \nabla _{g}u\right\vert ^{2}+\frac{%
b(x)}{\rho ^{\mu }}u^{2}\right) dv_{g}-\frac{\lambda }{q}\int_{M}\left\vert
u\right\vert ^{q}dv_{g}-\frac{1}{N}\int_{M}f(x)\left\vert u\right\vert
^{N}dv_{g}\text{.}
\end{equation*}%
If we put 
\begin{equation*}
\Phi _{\lambda }(u)=\left\langle \nabla J_{\lambda }(u),\text{ }%
u\right\rangle
\end{equation*}%
we get%
\begin{equation*}
\Phi _{\lambda }(u)=\int_{M}\left( \Delta _{g}u\right) ^{2}-\frac{a(x)}{\rho
^{\sigma }}\left\vert \nabla _{g}u\right\vert ^{2}+\frac{b(x)}{\rho ^{\mu }}%
u^{2}dv_{g}-\lambda \int_{M}\left\vert u\right\vert
^{q}dv_{g}-\int_{M}f(x)\left\vert u\right\vert ^{N}dv_{g}.
\end{equation*}

\begin{theorem}
\label{thm4} Let $0<\sigma <\frac{n}{s}<2$ and $0<\mu <\frac{n}{p}<4$.
Suppose that 
\begin{equation*}
\sup_{u\in H_{2}^{2}\left( M\right) }J_{\lambda ,\sigma ,\mu }(u)<\frac{2}{n%
\text{ }K_{o}^{\frac{n}{4}}(f(x_{\circ }))^{\frac{n-4}{4}}}
\end{equation*}%
then there is $\lambda _{\ast }>0$ such that if $\lambda \in (0,$ $\lambda
_{\ast })$, the equation (\ref{12}) possesses a weak non trivial solution $%
u_{\sigma ,\mu }\in M_{\lambda }$.
\end{theorem}

\begin{proof}
Let $\tilde{a}=\frac{a(x)}{\rho ^{\sigma }}$ and $\tilde{b}=\frac{b(x)}{\rho
^{\mu }}$, so if $\sigma \in (0,\min \left( 2,\frac{n}{s}\right) )$ and $\mu
\in (0,\min (4,\frac{n}{p}))$, obliviously $\tilde{a}\in L^{s}(M)$, $\tilde{b%
}\in L^{p}(M)$, where $s>\frac{n}{2}$ and $p>\frac{n}{4}$ .Theorem \ref{thm4}
is a consequence of Theorem \ref{thm3}.
\end{proof}

\section{The critical cases $\protect\sigma =2$ and $\protect\mu =4$}

By section four, for any $\sigma \in (0,\min \left( 2,\frac{n}{s}\right) )$
and $\mu \in (0,\min (4,\frac{n}{p}))$, there is a solution $u_{\sigma ,\mu
}\in M_{\lambda }$ of equation (\ref{2}). Now we are going to show that the
sequence $\left( u_{\sigma ,\mu }\right) _{\sigma ,\mu }$ is bounded in $%
H_{2}^{2}\left( M\right) $. Evaluating $J_{\lambda ,\sigma ,\mu }$ at $%
u_{\sigma ,\mu }$%
\begin{equation*}
J_{\lambda ,\sigma ,\mu }(u_{\sigma ,\mu })=\frac{1}{2}\left\Vert u_{\sigma
,\mu }\right\Vert ^{2}-\frac{1}{N}\int_{M}f(x)\left\vert u_{\sigma ,\mu
}\right\vert ^{N}dv_{g}-\frac{1}{q}\lambda \int_{M}\left\vert u_{\sigma ,\mu
}\right\vert ^{q}dv_{g}
\end{equation*}%
and taking account of $u_{\sigma ,\mu }\in M_{\lambda }$, we infer that%
\begin{equation*}
J_{\lambda ,\sigma ,\mu }(u_{\sigma ,\mu })=\frac{N-2}{2N}\left\Vert
u_{\sigma ,\mu }\right\Vert ^{2}-\lambda \frac{N-q}{Nq}\int_{M}\left\vert
u_{\sigma ,\mu }\right\vert ^{q}dv_{g}\text{.}
\end{equation*}

For a smooth function $a$ on $M$, denotes by $a^{-}=\min \left( 0,\min_{x\in
M}(a(x)\right) $. Let $K(n,2,\sigma )$ the best constant and $A\left(
\varepsilon ,\sigma \right) $ the corresponding constant in the Hardy-
Sobolev inequality given in Theorem \ref{th7}.

\begin{theorem}
Let $(M,g)$ be a Riemannian compact manifold of dimension $n\geq 5$. Let $%
\left( u_{m}\right) _{m}=\left( u_{\sigma _{m},\mu _{m}}\right) _{m}$ be a
sequence in $M_{\lambda }$ such that 
\begin{equation*}
\left\{ 
\begin{array}{c}
J_{\lambda ,\sigma ,\mu }(u_{m})\leq c_{\sigma ,\mu } \\ 
\nabla J_{\lambda }(u_{m})-\mu _{_{\sigma ,\mu }}\nabla \Phi _{\lambda
}(u_{m})\rightarrow 0%
\end{array}%
\right. \text{.}
\end{equation*}%
Suppose that 
\begin{equation*}
c_{\sigma ,\mu }<\frac{2}{n\text{ }K\left( n,2\right) ^{\frac{n}{4}%
}(\max_{x\in M}f(x))^{\frac{n-4}{4}}}
\end{equation*}%
and 
\begin{equation*}
1+a^{-}\max \left( K(n,2,\sigma ),A\left( \varepsilon ,\sigma \right)
\right) +b^{-}\max \left( K(n,2,\mu ),A\left( \varepsilon ,\mu \right)
\right) >0
\end{equation*}%
then the equation%
\begin{equation*}
\Delta ^{2}u-\nabla ^{\mu }(\frac{a}{\rho ^{2}}\nabla _{\mu }u)+\frac{bu}{%
\rho ^{4}}=f\left\vert u\right\vert ^{N-2}u+\lambda \left\vert u\right\vert
^{q-2}u
\end{equation*}%
has a non trivial solution in the distribution.
\end{theorem}

\begin{proof}
Let $\left( u_{m}\right) _{m}\subset M_{\lambda ,\sigma ,\mu }$,%
\begin{equation*}
J_{\lambda ,\sigma ,\mu }(u_{m})=\frac{N-2}{2N}\left\Vert u_{m}\right\Vert
^{2}-\lambda \frac{N-q}{Nq}\int_{M}\left\vert u_{m}\right\vert ^{q}dv_{g}
\end{equation*}%
As in proof of Theorem \ref{thm3}, we get 
\begin{equation*}
J_{\lambda ,\sigma ,\mu }(u_{m})\geq \left\Vert u_{m}\right\Vert ^{2}(\frac{%
N-2}{2N}-\lambda \frac{N-q}{Nq}\Lambda _{\sigma ,\mu }^{-\frac{q}{2}}V(M)^{1-%
\frac{q}{N}}(\max ((1+\varepsilon )K\left( n,2\right) ,A_{\varepsilon }))^{%
\frac{q}{2}}\tau ^{q-2})>0
\end{equation*}%
where $0<\lambda <\frac{\frac{\left( N-2\right) q}{2\left( N-q\right) }%
\Lambda _{\sigma ,\mu }^{\frac{q}{2}}}{V(M)^{1-\frac{q}{N}}(\max
((1+\varepsilon )K\left( n,2\right) ,A_{\varepsilon }))^{\frac{q}{2}}\tau
^{q-2}}$.

First we claim that%
\begin{equation*}
\lim_{\left( \sigma ,\mu \right) \rightarrow \left( 2^{-},4^{-}\right) }\inf
\Lambda _{\sigma ,\mu }>0\text{.}
\end{equation*}%
Indeed, if $\nu _{1,\sigma ,\mu }$ denotes the first nonzero eigenvalue of
the operator $P_{g}=\Delta _{g}^{2}-div\left( \frac{a}{\rho ^{\sigma }}%
\nabla _{g}\right) +\frac{b}{\rho ^{\mu }}$, then clearly $\Lambda _{\sigma
,\mu }\geq \nu _{1,\sigma ,\mu }$. Suppose by absurd that $\lim_{\left(
\sigma ,\mu \right) \rightarrow \left( 2^{-},4^{-}\right) }\inf \Lambda
_{\sigma ,\mu }=0$, then $\lim \inf_{\left( \sigma ,\mu \right) \rightarrow
\left( 2^{-},4^{-}\right) }\nu _{1,\sigma ,\mu }=0$. Independently, if $%
u_{\sigma ,\mu }$ is the corresponding eigenfunction to $\nu _{1,\sigma ,\mu
}$ we have%
\begin{equation*}
\nu _{1,\sigma ,\mu }=\left\Vert \Delta u\right\Vert _{2}^{2}+\int_{M}\frac{%
a\left\vert \nabla u\right\vert ^{2}}{\rho ^{\sigma }}dv_{g}+\int_{M}\frac{%
bu^{2}}{\rho ^{\mu }}dv_{g}
\end{equation*}%
\begin{equation}
\geq \left\Vert \Delta u\right\Vert _{2}^{2}+a^{-}\int \frac{\left\vert
\nabla u\right\vert ^{2}}{\rho ^{\sigma }}dv_{g}+b^{-}\int_{M}\frac{u^{2}}{%
\rho ^{\mu }}dv_{g}  \tag{13}  \label{13}
\end{equation}%
where $a^{-}=\min \left( 0,\min_{x\in M}a(x)\right) $ and $b^{-}=\min \left(
0,\min_{x\in M}b(x)\right) $. The Hardy- Sobolev's inequality given by
Theorem \ref{th7} leads to%
\begin{equation*}
\int_{M}\frac{\left\vert \nabla u\right\vert ^{2}}{\rho ^{\sigma }}%
dv_{g}\leq C(\left\Vert \nabla \left\vert \nabla u\right\vert \right\Vert
^{2}+\left\Vert \nabla u\right\Vert ^{2})
\end{equation*}%
and since 
\begin{equation*}
\left\Vert \nabla \left\vert \nabla u\right\vert \right\Vert ^{2}\leq
\left\Vert \nabla ^{2}u\right\Vert ^{2}\leq \left\Vert \Delta u\right\Vert
^{2}+\beta \left\Vert \nabla u\right\Vert ^{2}
\end{equation*}%
where $\beta >0$ is a constant and it is well known that for any $%
\varepsilon >0$ there is a constant $c\left( \varepsilon \right) >0$ such
that 
\begin{equation*}
\left\Vert \nabla u\right\Vert ^{2}\leq \varepsilon \left\Vert \Delta
u\right\Vert ^{2}+c\left\Vert u\right\Vert ^{2}\text{.}
\end{equation*}%
Hence 
\begin{equation}
\int_{M}\frac{\left\vert \nabla u\right\vert ^{2}}{\rho ^{\sigma }}%
dv_{g}\leq C\left( 1+\varepsilon \right) \left\Vert \Delta u\right\Vert
^{2}+A\left( \varepsilon \right) \left\Vert u\right\Vert ^{2}  \tag{14}
\label{14}
\end{equation}%
Now if $K(n,2,\sigma )$ denotes the best constant in inequality (\ref{14})
we get for any $\varepsilon >0$%
\begin{equation}
\int_{M}\frac{\left\vert \nabla u\right\vert ^{2}}{\rho ^{\sigma }}%
dv_{g}\leq \left( K(n,2,\sigma )^{2}+\varepsilon \right) \left\Vert \Delta
u\right\Vert ^{2}+A\left( \varepsilon ,\sigma \right) \left\Vert
u\right\Vert ^{2}\text{.}  \tag{15}  \label{15}
\end{equation}%
By the inequalities (\ref{11}), (\ref{13}) and (\ref{15}), we have%
\begin{equation*}
\nu _{1,\sigma ,\mu }\geq \left( 1+a^{-}\max \left( K(n,2,\sigma ),A\left(
\varepsilon ,\sigma \right) \right) +b^{-}\max \left( K(n,2,\mu ),A\left(
\varepsilon ,\mu \right) \right) \right) \left( \left\Vert \Delta u_{\sigma
,\mu }\right\Vert ^{2}+\left\Vert u_{\sigma ,\mu }\right\Vert ^{2}\right)
\end{equation*}%
So if 
\begin{equation*}
1+a^{-}\max \left( K(n,2,\sigma ),A\left( \varepsilon ,\sigma \right)
\right) +b^{-}\max \left( K(n,2,\mu ),A\left( \varepsilon ,\mu \right)
\right) >0
\end{equation*}%
then we get $\lim_{\sigma ,\mu }\left( u_{\sigma ,\mu }\right) =0$ and $%
\left\Vert u_{\sigma ,\mu }\right\Vert =1$ a contradiction.\ \newline
The reflexivity of $H_{2}^{2}(M)$ and the compactness of the embedding $%
H_{2}^{2}(M)\subset H_{p}^{k}(M)$\ ( $k\ =0$,$1$; $p<N$ ), imply that up to
a subsequence we have

$u_{m}\rightarrow u$\ \ weakly in $H_{2}^{2}(M)$

$u_{m}\rightarrow u$ strongly in $L^{p}(M)$, $p<N$

$\nabla u_{m}\rightarrow \nabla u$ strongly in $L^{p}(M)$, $p<2^{\ast }=%
\frac{2n}{n-2}$

$u_{m}\rightarrow u$ a.e. in $M$.

The Br\'{e}zis-Lieb lemma allows us to write 
\begin{equation*}
\int_{M}\left( \Delta _{g}u_{m}\right) ^{2}dv_{g}=\int_{M}\left( \Delta
_{g}u\right) ^{2}dv_{g}+\int_{M}\left( \Delta _{g}(u_{m}-u)\right)
^{2}dv_{g}+o(1)
\end{equation*}%
and also 
\begin{equation*}
\int_{M}f(x)\left\vert u_{m}\right\vert ^{N}dv_{g}=\int_{M}f(x)\left\vert
u\right\vert ^{N}dv_{g}+\int_{M}f(x)\left\vert u_{m}-u\right\vert
^{N}dv_{g}+o(1)\text{.}
\end{equation*}%
Now by the boundedness of the sequence ($u_{m}$)$_{m}$, we have that $%
u_{m}\rightarrow u$\ \ weakly in $H_{2}^{2}(M)$,

$\nabla u_{m}\rightarrow \nabla u$\ \ weakly in $L^{2}(M,\rho ^{-2})$ and $%
u_{m}\rightarrow u$\ \ weakly in $L^{2}(M,\rho ^{-4})$\ i.e. for any $%
\varphi \in L^{2}(M)$%
\begin{equation*}
\int_{M}\frac{a(x)}{\rho ^{2}}\nabla u_{m}\nabla \varphi dv_{g}=\int_{M}%
\frac{a(x)}{\rho ^{2}}\nabla u\nabla \varphi dv_{g}+o(1)
\end{equation*}%
and%
\begin{equation*}
\int_{M}\frac{b(x)}{\rho ^{4}}u_{m}\varphi dv_{g}=\int_{M}\frac{b(x)}{\rho
^{4}}u\varphi dv_{g}+o(1)\text{.}
\end{equation*}%
For every $\phi \in H_{2}^{2}(M)$ we have 
\begin{equation}
\int_{M}\left( \Delta _{g}^{2}u_{m}+div_{g}\left( \frac{a(x)}{\rho ^{\sigma
_{m}}}\nabla _{g}u_{m}\right) +\frac{b(x)}{\rho ^{\delta _{m}}}u_{m}\right)
\phi dv_{g}=\int_{M}\left( \lambda \left\vert u_{m}\right\vert
^{q-2}u_{m}+f(x)\left\vert u_{m}\right\vert ^{N-2}u_{m}\right) \phi dv_{g}%
\text{.}  \tag{16}  \label{16}
\end{equation}%
By the weak convergence in $H_{2}^{2}(M)$ we have immediately 
\begin{equation*}
\int_{M}\phi \Delta _{g}^{2}u_{m}dv_{g}=\int_{M}\phi \Delta
_{g}^{2}udv_{g}+o(1)
\end{equation*}%
and 
\begin{equation*}
\int_{M}\left( \frac{a(x)}{\rho ^{\sigma _{m}}}\nabla _{g}u_{m}-\frac{a(x)}{%
\rho ^{2}}\nabla _{g}u\right) \phi dv_{g}=\int_{M}\left( \frac{a(x)}{\rho
^{\sigma _{m}}}\nabla _{g}u_{m}+\frac{a(x)}{\rho ^{2}}\left( \nabla
_{g}u_{m}-\nabla _{g}u_{m}\right) -\frac{a(x)}{\rho ^{2}}\nabla _{g}u\right)
\phi dv_{g}
\end{equation*}%
Then%
\begin{equation*}
\left\vert \int_{M}\left( \frac{a(x)}{\rho ^{\sigma _{m}}}\nabla _{g}u_{m}-%
\frac{a(x)}{\rho ^{2}}\nabla _{g}u\right) \phi dv_{g}\right\vert \leq
\end{equation*}%
\begin{equation*}
\left\vert \int_{M}\left( \frac{a(x)}{\rho ^{\sigma _{m}}}\nabla _{g}u_{m}-%
\frac{a(x)}{\rho ^{2}}\nabla _{g}u_{m}\right) \phi dv_{g}\right\vert
+\left\vert \int_{M}\left( \frac{a(x)}{\rho ^{2}}\nabla _{g}u_{m}-\frac{a(x)%
}{\rho ^{2}}\nabla _{g}u\right) \phi dv_{g}\right\vert
\end{equation*}%
\begin{equation}
\leq \int_{M}\left\vert a(x)\phi \nabla _{g}u_{m}\right\vert \left\vert 
\frac{1}{\rho ^{\sigma _{m}}}-\frac{1}{\rho ^{2}}\right\vert
dv_{g}+\left\vert \int_{M}\frac{a(x)}{\rho ^{2}}\nabla _{g}\left(
u_{m}-u\right) \phi dv_{g}\right\vert \text{.}  \tag{17}  \label{17}
\end{equation}%
The weak convergence in $L^{2}(M,\rho ^{-2})$ and the Lebesgue's dominated
convergence theorem imply that the second right hand side of (\ref{17}) goes
to $0$. For the third term of the left hand side of (\ref{15}), we write 
\begin{equation*}
\int_{M}\left( \frac{b(x)}{\rho ^{\delta _{m}}}u_{m}-\frac{b(x)}{\rho ^{4}}%
u\right) \phi dv_{g}=\int_{M}\left( \frac{b(x)}{\rho ^{\delta _{m}}}u_{m}-%
\frac{b(x)}{\rho ^{4}}u_{m}+\frac{b(x)}{\rho ^{4}}u_{m}-\frac{b(x)}{\rho ^{4}%
}u\right) \phi dv_{g}
\end{equation*}%
and%
\begin{equation*}
\left\vert \int_{M}\left( \frac{b(x)}{\rho ^{\delta _{m}}}u_{m}-\frac{b(x)}{%
\rho ^{4}}u\right) \phi dv_{g}\right\vert
\end{equation*}%
\begin{equation}
\leq \int_{M}\left\vert b(x)\phi u_{m}\right\vert \left\vert \frac{1}{\rho
^{\delta _{m}}}-\frac{1}{\rho ^{4}}\right\vert dv_{g}+\left\vert \int_{M}%
\frac{b(x)}{\rho ^{4}}\left( u_{m}-u\right) \phi dv_{g}\right\vert \text{.} 
\tag{18}  \label{18}
\end{equation}%
Here also the weak convergence in $L^{2}(M,\rho ^{-4})$ and the Lebesgue's
dominated convergence allows us to affirm that the left hand side of (\ref%
{18}) converges to $0$.

It remains to show that $\mu _{m}\rightarrow 0$ as $m$ $\rightarrow +\infty $
and $u_{m}\rightarrow u$ strongly in $H_{2}^{2}\left( M\right) $ but this is
the same as in the proof of Theorem \ref{thm3} which implies also $u\in
M_{\lambda }.$
\end{proof}

\section{Test Functions}

In this section, we give the proof of the main theorem to do so, we consider
a normal geodesic coordinate system centred at $x_{o}$.\ Denote by $%
S_{x_{o}}(\rho )$ the geodesic sphere centred at $x_{o}$ and of radius $\rho 
$ ( $\rho <d$ =the injectivity radius). Let $d\Omega $ be the volume element
of the $n-1$-dimensional Euclidean unit sphere $S^{n-1}$ and put 
\begin{equation*}
G(\rho )=\frac{1}{\omega _{n-1}}\dint\limits_{S(\rho )}\sqrt{\left\vert
g(x)\right\vert }d\Omega
\end{equation*}%
where $\omega _{n-1}$ is the volume of $S^{n-1}$ and $\left\vert
g(x)\right\vert $ the determinant of the Riemannian metric $g$. The Taylor's
expansion of $G(\rho )$ in a neighborhood of $x_{o}$ is given by 
\begin{equation*}
G(\rho )=1-\frac{S_{g}(x_{\circ })}{6n}\rho ^{2}+o(\rho ^{2})
\end{equation*}%
where $S_{g}(x_{\circ })$ denotes the scalar curvature of $M$ at $x_{\circ }$%
.$\newline
$Let $B(x_{\circ },\delta )$ be the geodesic ball centred at $x_{\circ }$
and of radius $\delta $ such that $0<2\delta <d$ and denote by $\eta $ a
smooth function on $M$ such that 
\begin{equation*}
\eta (x)=\left\{ 
\begin{array}{c}
1\text{ \ \ \ \ \ \ \ \ on }B(x_{o},\delta ) \\ 
0\text{ \ \ on }M-B(x_{o},2\delta )%
\end{array}%
\right. \text{.}
\end{equation*}%
Consider the following radial function

\begin{equation*}
u_{\epsilon }(x)=(\frac{(n-4)n(n^{2}-4)\epsilon ^{4}}{f(x_{\circ })})^{\frac{%
n-4}{8}}\frac{\eta (\rho )}{(\left( \rho \theta \right) ^{2}+\epsilon ^{2})^{%
\frac{n-4}{2}}}
\end{equation*}%
with%
\begin{equation*}
\theta =\left( 1+\left\Vert a\right\Vert _{r}+\left\Vert b\right\Vert
_{s}\right) ^{\frac{1}{n}}
\end{equation*}%
where $\rho =d(x_{o},x)$ is the distance from $x_{o}$ to $x$ and $f(x_{\circ
})=\max_{x\in M}f(x)$. For further computations we need the following
integrals: for any real positive numbers $p$, $g$ such that $p-q>1$ we put

\begin{equation*}
I_{p}^{q}=\int_{0}^{+\infty }\frac{t^{q}}{(1+t)^{p}}dt
\end{equation*}%
and the following relations are immediate 
\begin{equation*}
I_{p+1}^{q}=\frac{p-q-1}{p}I_{p}^{q}\text{ \ \ \ and\ \ \ \ \ }I_{p+1}^{q+1}=%
\frac{q+1}{p-q-1}I_{p+1}^{q}\text{.}
\end{equation*}

\section{Application to compact Riemannian manifolds of dimension $n>6$}

\begin{theorem}
\label{thm5} Let $\left( M,g\right) $ be a compact Riemannian manifold of
dimension $\ n>6$. Suppose that at a point $x_{o}$ where $f$ attains its
maximum the following condition%
\begin{equation*}
\frac{\Delta f(x_{o})}{f\left( x_{o}\right) }<\frac{1}{3}\left( \frac{%
(n-1)n\left( n^{2}+4n-20\right) }{\left( n^{2}-4\right) \left( n-4\right)
\left( n-6\right) }\frac{1}{\left( 1+\left\Vert a\right\Vert _{r}+\left\Vert
b\right\Vert _{s}\right) ^{\frac{4}{n}}}-1\right) S_{g}\left( x_{o}\right)
\end{equation*}%
holds . Then the equation (\ref{1}) has a non trivial solution with energy%
\begin{equation*}
J_{\lambda }(u)<\frac{1}{K_{\circ }^{\frac{n}{4}}\left( \max_{x\in
M}f(x)\right) ^{\frac{n}{4}-1}}\text{.}
\end{equation*}
\end{theorem}

\begin{proof}
The proof of Theorem \ref{thm5} reduces to show that the condition (\ref{C1}%
) of Theorem \ref{thm3} is satisfied and since by Lemma \ref{lem} there is a 
$t_{o}>0$ such that $t_{o}u_{\epsilon }\in M_{\lambda }$ for sufficiently
small $\lambda $, so it suffices to show that%
\begin{equation*}
\sup_{t>0}J_{\lambda }\left( tu_{\epsilon }\right) <\frac{1}{K_{\circ }^{%
\frac{n}{4}}\left( \max_{x\in M}f(x)\right) ^{\frac{n}{4}-1}}\text{.}
\end{equation*}%
To compute the term $\int_{M}f(x)\left\vert u_{\epsilon }(x)\right\vert
^{N}dv_{g}$, we need the following Taylor's expansion of $f$ at the point $%
x_{o}$ 
\begin{equation*}
f(x)=f(x_{\circ })+\frac{\partial ^{2}f(x_{\circ })}{2\partial y^{i}\partial
y^{j}}y^{i}y^{j}+o(\rho ^{2})
\end{equation*}%
and also that of the Riemannian measure%
\begin{equation*}
dv_{g}=1-\frac{1}{6}R_{ij}(x_{\circ })y^{i}y^{j}+o(\rho ^{2})
\end{equation*}%
where $R_{ij}(x_{\circ })$ denotes the Ricci tensor at $x_{\circ }$. The
expression of $\int_{M}f(x)\left\vert u_{\epsilon }(x)\right\vert ^{N}dv_{g}$
is well known (see for example \cite{11} ) and is given in case $n>6$ by%
\begin{equation*}
\int_{M}f(x)\left\vert u_{\epsilon }(x)\right\vert ^{N}dv_{g}=\frac{\theta
^{-n}}{K_{\circ }^{\frac{n}{4}}(f(x_{\circ }))^{\frac{n-4}{4}}}\left( 1-(%
\frac{\Delta f(x_{\circ })}{2(n-2)f(x_{\circ })}+\frac{S_{g}(x_{\circ })}{%
6(n-2)})\epsilon ^{2}+o(\epsilon ^{2})\right)
\end{equation*}%
where $K_{o}$ is given by (\ref{K}) and $\omega _{n}=2^{n-1}I_{n}^{\frac{n}{2%
}-1}\omega _{n-1}$ and $\omega _{n}$ is the volume of $S^{n}$, the standard
unit sphere of $R^{n+1}$ endowed with its round metric.

Now the restriction of $\left\vert \frac{\partial u_{\epsilon }}{\partial
\rho }\right\vert $ to the geodesic ball $B(x_{\circ },\delta )$ is computed
as follows 
\begin{equation*}
\left\vert \frac{\partial u_{\epsilon }}{\partial \rho }\right\vert
_{B(x_{\circ },\delta )}=\left\vert \nabla u_{\epsilon }\right\vert =\theta
^{-2}(n-4)\left( \frac{(n-4)n(n^{2}-4)\epsilon ^{4}}{f(x_{\circ })}\right) ^{%
\frac{n-4}{8}}\frac{\rho }{(\left( \frac{\rho }{\theta }\right)
^{2}+\epsilon ^{2})^{\frac{n-2}{2}}}
\end{equation*}%
and Since $a\in L^{r}(M)$ with $r>\frac{n}{2}$ we have 
\begin{equation*}
\int_{B(x_{\circ },\delta )}a(x)\left\vert \nabla u_{\epsilon }\right\vert
^{2}dv_{g}\leq \theta ^{-4}(n-4)^{2}\left( \frac{(n-4)n(n^{2}-4)\epsilon ^{4}%
}{f(x_{\circ })}\right) ^{\frac{n-4}{4}}\left\Vert a\right\Vert _{r}\omega
_{n-1}^{1-\frac{1}{r}}
\end{equation*}%
\begin{equation*}
\times \left( \int_{0}^{\delta }\frac{\rho ^{\frac{2r}{r-1}+n-1}}{(\left( 
\frac{\rho }{\theta }\right) ^{2}+\epsilon ^{2})^{\frac{\left( n-2\right) r}{%
r-1}}}\left( \int_{S(\rho )}\sqrt{\left\vert g(x)\right\vert }d\Omega
\right) d\rho \right) ^{\frac{r-1}{r}}
\end{equation*}%
Since 
\begin{equation*}
\dint\limits_{S(\rho )}\sqrt{\left\vert g(x)\right\vert }d\Omega =\omega
_{n-1}\left( 1-\frac{S_{g}(x_{\circ })}{6n}\rho ^{2}+o(\rho ^{2})\right)
\end{equation*}%
we get%
\begin{equation*}
\int_{B(x_{\circ },\delta )}a(x)\left\vert \nabla u_{\epsilon }\right\vert
^{2}dv_{g}\leq \theta ^{-4}(n-4)^{2}\left( \frac{(n-4)n(n^{2}-4)\epsilon ^{4}%
}{f(x_{\circ })}\right) ^{\frac{n-4}{4}}\left\Vert a\right\Vert _{r}\omega
_{n-1}^{1-\frac{1}{r}}
\end{equation*}%
\begin{equation*}
\times \left( \int_{0}^{\delta }\frac{\rho ^{\frac{2r}{r-1}+n-1}}{(\left(
\rho \theta \right) ^{2}+\epsilon ^{2})^{\frac{\left( n-2\right) r}{r-1}}}%
d\rho \left( 1-\frac{S_{g}(x_{\circ })}{6n}\rho ^{2}+o(\rho ^{2})\right)
\right) ^{\frac{r-1}{r}}
\end{equation*}%
and by the following change of variable 
\begin{equation*}
t=(\frac{\rho \theta }{\epsilon })^{2}\text{ i.e. \ }\rho =\frac{\epsilon }{%
\theta }\sqrt{t}
\end{equation*}%
we obtain%
\begin{equation*}
\int_{B(x_{\circ },\delta )}a(x)\left\vert \nabla u_{\epsilon }\right\vert
^{2}dv_{g}\leq \theta ^{-n\frac{r}{r-1}}(n-4)^{2}\left( \frac{%
(n-4)n(n^{2}-4)\epsilon ^{4}}{f(x_{\circ })}\right) ^{\frac{n-4}{4}%
}\left\Vert a\right\Vert _{r}\omega _{n-1}^{1-\frac{1}{r}}\epsilon ^{-\left(
n-4\right) +2-\frac{n}{r}}
\end{equation*}%
\begin{equation*}
\times \left( \int_{0}^{(\frac{\delta \theta }{\epsilon })^{2}}\frac{t^{%
\frac{n-2}{2}+\frac{r}{r-1}}}{(t+1)^{\frac{\left( n-2\right) r}{r-1}}}dt-%
\frac{S_{g}(x_{\circ })}{6n}\theta ^{-2}\epsilon ^{2}\int_{0}^{(\frac{\delta
\theta }{\epsilon })^{2}}\frac{t^{\frac{n}{2}+\frac{r}{r-1}}}{(t+1)^{\frac{%
\left( n-2\right) r}{r-1}}}dt+o(\epsilon ^{2})\right) ^{\frac{r-1}{r}}
\end{equation*}%
letting $\epsilon \rightarrow 0$ we get%
\begin{equation*}
\int_{B(x_{\circ },\delta )}a(x)\left\vert \nabla u_{\epsilon }\right\vert
^{2}dv_{g}\leq 2^{-1+\frac{1}{r}}\theta ^{-n\left( 1-\frac{1}{r}\right)
}(n-4)^{2}\left( \frac{(n-4)n(n^{2}-4)\epsilon ^{4}}{f(x_{\circ })}\right) ^{%
\frac{n-4}{4}}\left\Vert a\right\Vert _{r}\omega _{n-1}^{1-\frac{1}{r}%
}\epsilon ^{-\left( n-4\right) +2-\frac{n}{r}}
\end{equation*}%
\begin{equation*}
\times \left( I_{\frac{\left( n-2\right) r}{r-1}}^{\frac{n-2}{2}+\frac{r}{r-1%
}}-\theta ^{-2}\frac{S_{g}(x_{\circ })}{6n}I_{\frac{\left( n-2\right) r}{r-1}%
}^{\frac{n}{2}+\frac{r}{r-1}}\epsilon ^{2}+o(\epsilon ^{2})\right) ^{\frac{%
r-1}{r}}
\end{equation*}%
Then%
\begin{equation*}
\int_{B(x_{\circ },\delta )}a(x)\left\vert \nabla u_{\epsilon }\right\vert
^{2}dv_{g}\leq 2^{-1+\frac{1}{r}}\theta ^{-n\frac{r}{r-1}}(n-4)^{2}\left( 
\frac{(n-4)n(n^{2}-4)\epsilon ^{4}}{f(x_{\circ })}\right) ^{\frac{n-4}{4}%
}\left\Vert a\right\Vert _{r}\omega _{n-1}^{1-\frac{1}{r}}\epsilon
^{\epsilon ^{-\left( n-4\right) +2-\frac{n}{r}}}
\end{equation*}%
\begin{equation*}
\times I_{\frac{\left( n-2\right) r}{r-1}}^{1+\frac{n-2}{2}.\frac{r-1}{r}}%
\left[ 1-\frac{r-1}{r}\theta ^{2}\frac{S_{g}(x_{\circ })}{6n}I_{\frac{\left(
n-2\right) r}{r-1}}^{\frac{n}{2}+\frac{r}{r-1}}I_{\frac{\left( n-2\right) r}{%
r-1}}^{-\frac{n-2}{2}-\frac{r}{r-1}}\epsilon ^{2}+o(\epsilon ^{2})\right] 
\text{.}
\end{equation*}%
It remains to compute the integral $\int_{B(x_{\circ },2\delta )-B(x_{\circ
},\delta )}a(x)\left\vert \nabla u_{\epsilon }\right\vert ^{2}dv_{g}$.%
\newline
First we remark that 
\begin{equation*}
\left\vert \int_{(\frac{\delta \theta }{\epsilon })^{2}}^{(\frac{2\delta
\theta }{\epsilon })^{2}}h(t)\frac{t^{q}}{(t+1)^{p}}dt\right\vert \leq
C\left( \frac{1}{\epsilon }\right) ^{2(q-p+1)}=C\epsilon ^{2(p-q-1)}
\end{equation*}%
and since $p-q=n-4\geq 3$, we obtain%
\begin{equation*}
\int_{(\frac{\delta \theta }{\epsilon })^{2}}^{(\frac{2\delta \theta }{%
\epsilon })^{2}}h(t)\frac{i^{q}}{(t+1)^{p}}dt=o(\epsilon ^{2})
\end{equation*}%
and then 
\begin{equation}
\int_{B(x_{\circ },2\delta )-B(x_{\circ },\delta )}a(x)\left\vert \nabla
u_{\epsilon }\right\vert ^{2}dv_{g}=o(\epsilon ^{2})\text{.}  \tag{19}
\label{19}
\end{equation}%
Finally we get $\ $%
\begin{equation*}
\int_{M}a(x)\left\vert \nabla u_{\epsilon }\right\vert ^{2}dv_{g}\leq 2^{-1+%
\frac{1}{r}}\theta ^{-n\frac{r}{r-1}}(n-4)^{2}\left( \frac{%
(n-4)n(n^{2}-4)\epsilon ^{4}}{f(x_{\circ })}\right) ^{\frac{n-4}{4}%
}\left\Vert a\right\Vert _{r}\omega _{n-1}^{1-\frac{1}{r}}\epsilon ^{-\left(
n-4\right) +2-\frac{n}{r}}
\end{equation*}%
\begin{equation*}
\times \left( I_{\frac{\left( n-2\right) r}{r-1}}^{1+\frac{n-2}{2}.\frac{r-1%
}{r}}+o(\epsilon ^{2})\right) \text{.}
\end{equation*}%
Letting 
\begin{equation}
A=K_{\circ }^{\frac{n}{4}}\frac{(n-4)^{\frac{n}{4}+1}\times \left( \omega
_{n-1}\right) ^{\frac{r-1}{r}}}{2^{\frac{r-1}{r}}}(n(n^{2}-4))^{\frac{n-4}{4}%
}\left( I_{\frac{\left( n-2\right) r}{r-1}}^{\frac{n-2}{2}+\frac{r}{r-1}%
}\right) ^{\frac{r-1}{r}}  \tag{20}  \label{20}
\end{equation}%
we obtain 
\begin{equation*}
\int_{M}a(x)\left\vert \nabla u_{\epsilon }\right\vert ^{2}dv_{g}\leq
\epsilon ^{2-\frac{n}{r}}\theta ^{-n\frac{r}{r-1}}\frac{A}{\text{ }K_{\circ
}^{\frac{n}{4}}(f(x_{\circ }))^{\frac{n-4}{4}}}\left\Vert a\right\Vert
_{r}\left( 1+o(\epsilon ^{2})\right) \text{.}
\end{equation*}%
Now we compute%
\begin{equation*}
\int_{M}b(x)u_{\epsilon }^{2}dv_{g}=\int_{B(x_{\circ },\delta
)}b(x)u_{\epsilon }^{2}dv_{g}+\int_{B(x_{\circ },2\delta )-B(x_{\circ
},\delta )}b(x)u_{\epsilon }^{2}dv_{g}
\end{equation*}%
and since $b\in L^{s}(M)$ with $s>\frac{n}{4}$, we have 
\begin{equation*}
\int_{M}b(x)u_{\epsilon }^{2}dv_{g}\leq \left\Vert b\right\Vert
_{s}\left\Vert u_{\epsilon }\right\Vert _{\frac{2s}{s-1}}^{2}\text{.}
\end{equation*}

Independently%
\begin{equation*}
=\left( \frac{\left( n-4\right) n\left( n^{2}-4\right) \epsilon ^{4}}{%
f(x_{o})}\right) ^{\frac{n-4}{4}}
\end{equation*}%
\begin{equation*}
\left\Vert u_{\epsilon }\right\Vert _{\frac{2s}{s-1},B(x_{o},\delta
)}^{2}=\left( \frac{(n-4)n(n^{2}-4)\epsilon ^{4}}{f(x_{\circ })}\right) ^{%
\frac{n-4}{4}}\left( \int_{0}^{\delta }\frac{\rho ^{n-1}}{(\left( \rho
\theta \right) ^{2}+\epsilon ^{2})^{\frac{\left( n-4\right) s}{(s-1)}}}%
\left( \int_{S(r)}\sqrt{\left\vert g(x)\right\vert }d\Omega \right)
dr\right) ^{\frac{s-1}{s}}
\end{equation*}%
and%
\begin{equation*}
\dint\limits_{S(r)}\sqrt{\left\vert g(x)\right\vert }d\Omega =\omega
_{n-1}\left( 1-\frac{S_{g}(x_{\circ })}{6n}\rho ^{2}+o(\rho ^{2})\right) 
\text{.}
\end{equation*}%
consequently%
\begin{equation*}
\left\Vert u_{\epsilon }\right\Vert _{\frac{2s}{s-1},B(x_{o},\delta
)}^{2}=\left( \frac{(n-4)n(n^{2}-4)\epsilon ^{4}}{f(x_{\circ })}\right) ^{%
\frac{n-4}{4}}\omega _{n-1}^{\frac{s-1}{s}}\times
\end{equation*}%
\begin{equation*}
\left( \int_{0}^{\delta }\frac{\rho ^{n-1}}{(\left( \rho \theta \right)
^{2}+\epsilon ^{2})^{\frac{\left( n-4\right) s}{(s-1)}}}\left( 1-\frac{%
S_{g}(x_{\circ })}{6n}\rho ^{2}+o(\rho ^{2})\right) d\rho \right) ^{\frac{s-1%
}{s}}\text{.}
\end{equation*}%
And putting $t=(\frac{\rho \theta }{\epsilon })^{2}$ , we get 
\begin{equation*}
\left\Vert u_{\epsilon }\right\Vert _{\frac{2s}{s-1},B(x_{o},\delta
)}^{2}=\left( \frac{(n-4)n(n^{2}-4)\epsilon ^{4}}{f(x_{\circ })}\right) ^{%
\frac{n-4}{4}}\left( \omega _{n-1}\right) ^{\frac{s-1}{s}}\epsilon ^{-n+4+4-%
\frac{n}{s}}\times
\end{equation*}%
\begin{equation*}
\left( \frac{\epsilon ^{n}\theta ^{-n}}{2}\int_{0}^{(\frac{\delta \theta }{%
\epsilon })^{2}}\frac{t^{\frac{n}{2}-1}}{(t+1)^{\frac{\left( n-4\right) s}{%
(s-1)}}}dt-\frac{\theta ^{-n-2}S_{g}(x_{\circ })}{12n}\epsilon
^{n+2}\int_{0}^{(\frac{\delta \theta }{\epsilon })^{2}}\frac{t^{\frac{n}{2}}%
}{(t+1)^{\frac{\left( n-4\right) s}{(s-1)}}}dt+o(\epsilon ^{n+2})\right) ^{%
\frac{s-1}{s}}\text{.}
\end{equation*}%
Letting $\epsilon \rightarrow 0$, we get%
\begin{equation*}
\left\Vert u_{\epsilon }\right\Vert _{\frac{2s}{s-1},B\left( x_{o},\delta
\right) }^{2}=\left( \frac{(n-4)n(n^{2}-4)\epsilon ^{4}}{f(x_{\circ })}%
\right) ^{\frac{n-4}{4}}\left( \omega _{n-1}\right) ^{\frac{s-1}{s}}\epsilon
^{-n+4+4-\frac{n}{s}}
\end{equation*}%
\begin{equation*}
\times \theta ^{-n\frac{s}{s-1}}(\frac{\epsilon ^{n}}{2})^{\frac{s-1}{s}%
}\left( \int_{0}^{+\infty }\frac{t^{\frac{n}{2}}}{(t+1)^{\frac{\left(
n-4\right) s}{(s-1)}}}dt-\frac{S_{g}(x_{\circ })}{12n}\epsilon ^{2}\theta
^{-2}\int_{0}^{+\infty }\frac{t^{\frac{n}{2}+1}}{(t+1)^{\frac{\left(
n-4\right) s}{(s-1)}}}dt+o(\epsilon ^{2})\right) ^{\frac{s-1}{s}}\text{.}
\end{equation*}%
Hence 
\begin{equation*}
\left\Vert u_{\epsilon }\right\Vert _{\frac{2s}{s-1},B(x_{o},\delta
)}^{2}=\left( \frac{(n-4)n(n^{2}-4)\epsilon ^{4}}{f(x_{\circ })}\right) ^{%
\frac{n-4}{4}}\left( \omega _{n-1}\right) ^{\frac{s-1}{s}}\epsilon ^{-n+4+4-%
\frac{n}{s}}
\end{equation*}%
\begin{equation*}
\times \theta ^{-n\frac{s}{s-1}}(\frac{\epsilon ^{n}}{2})^{\frac{s-1}{s}%
}\left( \int_{0}^{+\infty }\frac{t^{\frac{n}{2}}}{(t+1)^{\frac{\left(
n-4\right) s}{(s-1)}}}dt-\theta ^{-2}\frac{S_{g}(x_{\circ })}{12n}\epsilon
^{2}\int_{0}^{+\infty }\frac{t^{\frac{n}{2}+1}}{(t+1)^{\frac{\left(
n-4\right) s}{(s-1)}}}dt+o(\epsilon ^{2})\right) ^{\frac{s-1}{s}}\text{.}
\end{equation*}%
Or%
\begin{equation*}
\left\Vert u_{\epsilon }\right\Vert _{\frac{2s}{s-1}}^{2}=\left( \frac{%
(n-4)n(n^{2}-4)}{f(x_{\circ })}\right) ^{\frac{n-4}{4}}\left( \frac{\omega
_{n-1}}{2}\right) ^{\frac{s-1}{s}}\epsilon ^{4-\frac{n}{s}}\theta ^{-n\frac{s%
}{s-1}}
\end{equation*}%
\begin{equation*}
\times \left[ \left( I_{\frac{\left( n-4\right) s}{(s-1)}}^{\frac{n}{2}%
}\right) ^{\frac{s-1}{s}}-\frac{\theta ^{-2}(s-1)S_{g}(x_{\circ })}{12n\text{
}s}\left( I_{\frac{\left( n-4\right) s}{(s-1)}}^{\frac{n}{2}}\right) ^{-%
\frac{1}{s}}I_{\frac{\left( n-4\right) s}{(s-1)}}^{\frac{n}{2}+1}\epsilon
^{2}+o(\epsilon ^{2})\right]
\end{equation*}%
Finally, by the same manner as in equality (\ref{19}) we get%
\begin{equation*}
\int_{M}b(x)u_{\epsilon }^{2}dv_{g}\leq \left\Vert b\right\Vert _{s}\left( 
\frac{(n-4)n(n^{2}-4)}{f(x_{\circ })}\right) ^{\frac{n-4}{4}}(\frac{\omega
_{n-1}}{2})^{\frac{s-1}{s}}\epsilon ^{4-\frac{n}{s}}\theta ^{-n\frac{s}{s-1}}
\end{equation*}%
\begin{equation*}
\times \left( \left( I_{\frac{\left( n-4\right) s}{(s-1)}}^{\frac{n}{2}%
}\right) ^{\frac{s-1}{s}}+o(\epsilon ^{2})\right)
\end{equation*}%
Putting 
\begin{equation}
B=K_{\circ }^{\frac{n}{4}}((n-4)n(n^{2}-4))^{\frac{n-4}{4}}(\frac{\omega
_{n-1}}{2})^{\frac{s-1}{s}}\left( I_{\frac{\left( n-4\right) s}{(s-1)}}^{%
\frac{n}{2}}\right) ^{\frac{s-1}{s}}  \tag{21}  \label{21}
\end{equation}%
we get%
\begin{equation*}
\int_{M}b(x)u_{\epsilon }^{2}dv_{g}\leq \epsilon ^{4-\frac{n}{s}}\theta ^{-n%
\frac{s}{s-1}}\frac{\left\Vert b\right\Vert _{s}B}{\text{ }K_{\circ }^{\frac{%
n}{4}}(f(x_{\circ }))^{\frac{n-4}{4}}}\left( 1+o(\epsilon ^{2})\right) \text{%
.}
\end{equation*}%
The computation of $\int_{M}\left( \Delta u_{\epsilon }\right) ^{2}dv_{g}$
is well known see for example (\cite{11})\ and is given by%
\begin{equation*}
\int_{M}\left( \Delta u_{\epsilon }\right) ^{2}dv_{g}=\frac{\theta ^{-n}}{%
K_{\circ }^{\frac{n}{4}}(f(x_{\circ }))^{\frac{n-4}{4}}}\left( 1-\frac{%
n^{2}+4n-20}{6(n^{2}-4)(n-6)}S_{g}(x_{\circ })\epsilon ^{2}+o(\epsilon
^{2})\right) \text{.}
\end{equation*}%
Resuming we get%
\begin{equation*}
\int_{M}\left( \Delta u_{\epsilon }\right) ^{2}-a(x)\left\vert \nabla
u_{\epsilon }\right\vert ^{2}+b(x)u_{\epsilon }^{2}dv_{g}\leq \frac{\theta
^{-n}}{K_{\circ }^{\frac{n}{4}}f(x_{\circ })^{\frac{n-4}{4}}}\times
\end{equation*}%
\begin{equation*}
\left( 1+\epsilon ^{2-\frac{n}{r}}\theta ^{-\frac{n}{r-1}}A\left\Vert
a\right\Vert _{r}+\epsilon ^{4-\frac{n}{s}}\theta ^{-\frac{n}{s-1}%
}B\left\Vert b\right\Vert _{s}-\frac{n^{2}+4n-20}{6(n^{2}-4)(n-6)}%
S_{g}(x_{\circ })\epsilon ^{2}+o(\epsilon ^{2})\right) \text{.}
\end{equation*}%
Now, we have 
\begin{equation*}
J_{\lambda }\left( tu_{\epsilon }\right) \leq J_{o}\left( tu_{\epsilon
}\right) =\frac{t^{2}}{2}\left\Vert u_{\epsilon }\right\Vert ^{2}-\frac{t^{N}%
}{N}\int_{M}f(x)\left\vert u_{\epsilon }(x)\right\vert ^{N}dv_{g}
\end{equation*}%
\begin{equation*}
\leq \frac{\theta ^{-n}}{K_{\circ }^{\frac{n}{4}}f(x_{\circ })^{\frac{n-4}{4}%
}}\left\{ \frac{1}{2}t^{2}\left( 1+\epsilon ^{2-\frac{n}{r}}\theta ^{-\frac{n%
}{r-1}}A\left\Vert a\right\Vert _{r}+\epsilon ^{4-\frac{n}{s}}\theta ^{-%
\frac{n}{s-1}}B\left\Vert b\right\Vert _{s}\right) -\frac{t^{N}}{N}\right.
\end{equation*}%
\begin{equation*}
\left. +\left[ \left( \frac{\Delta f(x_{o})}{2\left( n-2\right) f(x_{o})}+%
\frac{S_{g}\left( x_{o}\right) }{6\left( n-1\right) }\right) \frac{t^{N}}{N}-%
\frac{1}{2}t^{2}\frac{n^{2}+4n-20}{6\left( n^{2}-4\right) \left( n-6\right) }%
S_{g}\left( x_{o}\right) \right] \epsilon ^{2}\right\}
\end{equation*}%
\begin{equation*}
+o\left( \epsilon ^{2}\right)
\end{equation*}%
and letting $\epsilon $ small enough so that%
\begin{equation*}
1+\epsilon ^{2-\frac{n}{r}}\theta ^{-\frac{n}{r-1}}A\left\Vert a\right\Vert
_{r}+\epsilon ^{4-\frac{n}{s}}\theta ^{-\frac{n}{s-1}}B\left\Vert
b\right\Vert _{s}\leq \left( 1+\left\Vert a\right\Vert _{r}+\left\Vert
b\right\Vert _{s}\right) ^{\frac{4}{n}}
\end{equation*}%
and since the function $\varphi (t)=\alpha \frac{t^{2}}{2}-\frac{t^{N}}{N}$,
with $\alpha >0$ and $t>0$, attains its maximum at $t_{o}=\alpha ^{\frac{1}{%
N-2}}$ and%
\begin{equation*}
\varphi (t_{o})=\frac{2}{n}\alpha ^{\frac{n}{4}}\text{.}
\end{equation*}%
Consequently, we get%
\begin{equation*}
J_{\lambda }\left( tu_{\epsilon }\right) \leq \frac{2\theta ^{-n}}{nK_{\circ
}^{\frac{n}{4}}f(x_{\circ })^{\frac{n-4}{4}}}\left\{ 1+\left\Vert
a\right\Vert _{r}+\left\Vert b\right\Vert _{s}\right.
\end{equation*}%
\begin{equation*}
\left. +\left[ \left( \frac{\Delta f(x_{o})}{2\left( n-2\right) f(x_{o})}+%
\frac{S_{g}\left( x_{o}\right) }{6\left( n-1\right) }\right) \frac{t_{o}^{N}%
}{N}-\frac{1}{2}t_{o}^{2}\frac{n^{2}+4n-20}{6\left( n^{2}-4\right) \left(
n-6\right) }S_{g}\left( x_{o}\right) \right] \epsilon ^{2}\right\}
\end{equation*}%
\begin{equation*}
+o\left( \epsilon ^{2}\right) \text{.}
\end{equation*}%
Taking account of the value of $\theta $ and putting%
\begin{equation*}
R(t)=\left( \frac{\Delta f(x_{o})}{2\left( n-2\right) f(x_{o})}+\frac{%
S_{g}\left( x_{o}\right) }{6\left( n-1\right) }\right) \frac{t^{N}}{N}-\frac{%
1}{2}\frac{n^{2}+4n-20}{6\left( n^{2}-4\right) \left( n-6\right) }%
S_{g}\left( x_{o}\right) t^{2}
\end{equation*}%
we obtain 
\begin{equation*}
\sup_{t\geq 0}J_{\lambda }\left( tu_{\epsilon }\right) <\frac{2}{nK_{\circ
}^{\frac{n}{4}}\left( \max_{x\in M}f(x)\right) ^{\frac{n}{4}-1}}
\end{equation*}%
provided that $R(t_{o})<0$ i.e.%
\begin{equation*}
\frac{\Delta f(x_{o})}{f\left( x_{o}\right) }<\left( \frac{n\left(
n^{2}+4n-20\right) }{3\left( n+2\right) \left( n-4\right) \left( n-6\right) }%
\frac{1}{\left( 1+\left\Vert a\right\Vert _{r}+\left\Vert b\right\Vert
_{s}\right) ^{\frac{4}{n}}}-\frac{n-2}{3\left( n-1\right) }\right)
S_{g}\left( x_{o}\right) \text{.}
\end{equation*}%
Which completes the proof.
\end{proof}

\section{Application to compact Riemannian manifolds of dimension $n=6$}

\begin{theorem}
In case $n=6$, we suppose that at a point $x_{o}$ where $f$ attains its
maximum $S_{g}\left( x_{o}\right) >0$. Then the equation (\ref{1}) has a non
trivial solution.
\end{theorem}

\begin{proof}
The same calculations as in case $n>6$ gives us%
\begin{equation*}
\int_{M}f(x)\left\vert u_{\epsilon }(x)\right\vert ^{N}dv_{g}=\frac{\theta
^{-n}}{K_{\circ }^{\frac{n}{4}}(f(x_{\circ }))^{\frac{n-4}{4}}}\left( 1-(%
\frac{\Delta f(x_{\circ })}{2(n-2)f(x_{\circ })}+\frac{S_{g}(x_{\circ })}{%
6(n-2)})\epsilon ^{2}+o(\epsilon ^{2})\right) \text{.}
\end{equation*}%
Also, we have 
\begin{equation*}
\int_{M}a(x)\left\vert \nabla u_{\epsilon }\right\vert ^{2}dv_{g}\leq \frac{%
\left\Vert a\right\Vert _{r}A}{\text{ }K_{\circ }^{\frac{n}{4}}(f(x_{\circ
}))^{\frac{n-4}{4}}}\epsilon ^{2-\frac{n}{r}\theta ^{-\frac{r}{r-1}}}\left(
1+o\left( \epsilon ^{2}\right) \right)
\end{equation*}%
and 
\begin{equation*}
\int_{M}b(x)u_{\epsilon }^{2}dv_{g}\leq \frac{\left\Vert b\right\Vert _{s}B}{%
K_{\circ }^{\frac{n}{4}}(f(x_{\circ }))^{\frac{n-4}{4}}}\epsilon ^{4-\frac{n%
}{s}}\theta ^{-\frac{s}{s-1}}+\left( 1+o\left( \epsilon ^{2}\right) \right) 
\text{.}
\end{equation*}%
where $A$ and $B$ are given by (\ref{20}) and (\ref{21}) respectively for $%
n=6$. The computations of the term $\int_{M}\left( \Delta u_{\epsilon
}\right) ^{2}dv_{g}$ are well known ( see for example \cite{11})

\begin{equation*}
\int_{M}\left( \Delta u_{\epsilon }\right) ^{2}dv(g)=\theta ^{-n}(n-4)^{2}(%
\frac{(n-4)n(n^{2}-4)}{f(x_{\circ })})^{\frac{n-4}{4}}\frac{\omega _{n-1}}{2}
\end{equation*}%
\begin{equation*}
\times \left( \frac{n(n+2)(n-2)}{(n-4)}I_{n}^{\frac{n}{2}-1}-\frac{2}{n}%
\theta ^{-2}S_{g}(x_{\circ })\epsilon ^{2}\log (\frac{1}{\epsilon ^{2}}%
)+O(\epsilon ^{2})\right) \text{.}
\end{equation*}%
\begin{equation*}
\int_{M}\left( \Delta u_{\epsilon }\right) ^{2}dv_{g}=\frac{\theta ^{-n}}{%
K_{\circ }^{\frac{n}{4}}(f(x_{\circ }))^{\frac{n-4}{4}}}\left( 1-\frac{%
2\left( n-4\right) }{n^{2}(n^{2}-4)I_{n}^{\frac{n}{2}-1}}S_{g}(x_{\circ
})\epsilon ^{2}\log \left( \frac{1}{\epsilon ^{2}}\right) +O(\epsilon
^{2})\right) \text{.}
\end{equation*}%
Now resuming and letting $\epsilon $ so that 
\begin{equation*}
1+\epsilon ^{2-\frac{n}{r}}\theta ^{-\frac{n}{r-1}}A\left\Vert b\right\Vert
_{s}+\epsilon ^{4-\frac{n}{s}}\theta ^{-\frac{n}{s-1}}B\left\Vert
a\right\Vert _{r}\leq \left( 1+\left\Vert a\right\Vert _{r}+\left\Vert
b\right\Vert _{s}\right) ^{\frac{4}{n}}
\end{equation*}%
we get%
\begin{equation*}
J_{\lambda }\left( u_{\epsilon }\right) \leq \frac{1}{2}\left\Vert
u_{\epsilon }\right\Vert ^{2}-\frac{1}{N}\int_{M}f(x)\left\vert u_{\epsilon
}(x)\right\vert ^{N}dv_{g}
\end{equation*}%
\begin{equation*}
\leq \frac{\theta ^{-n}}{\text{ }K_{\circ }^{\frac{n}{4}}(f(x_{\circ }))^{%
\frac{n-4}{4}}}\left[ \frac{t^{2}}{2}\left( 1+\left\Vert a\right\Vert
_{r}+\left\Vert b\right\Vert _{s}\right) ^{1-\frac{4}{n}}-\frac{t^{N}}{N}%
\right.
\end{equation*}%
\begin{equation*}
\left. -\frac{n-4}{n^{2}\left( n^{2}-4\right) I_{n}^{\frac{n}{2}-1}}\theta
^{-2}S_{g}(x_{\circ })t^{2}\epsilon ^{2}\log \left( \frac{1}{\epsilon ^{2}}%
\right) \right] +O(\epsilon ^{2})\text{.}
\end{equation*}%
The same arguments as in the case $n>6$ allow us to infer that 
\begin{equation*}
\max_{t\geq 0}J_{\lambda }\left( tu_{\epsilon }\right) <\frac{2}{n\text{ }%
K_{\circ }^{\frac{n}{4}}(f(x_{\circ }))^{\frac{n-4}{4}}}
\end{equation*}%
if 
\begin{equation*}
S_{g}(x_{\circ })>0\text{.}
\end{equation*}%
Which achieves the proof.
\end{proof}

\end{document}